\documentclass[11pt,leqno]{article}

\usepackage{amssymb,amsmath,amsthm,mysects,url,rotating}
 \usepackage{graphicx,epsfig}
 \usepackage{xcolor,graphicx}
\newcommand{\n}{\noindent}

\newcommand{\vp}{\varepsilon}
\newcommand{\bb}[1]{\mathbb{#1}}
\newcommand{\cl}[1]{\mathcal{#1}}

\newcommand{\ovl}{\overline}

\theoremstyle{plain}
\newtheorem{thm}{Theorem}[section]
\newtheorem{lem}[thm]{Lemma}

\newtheorem{pro}[thm]{Proposition}

\newtheorem{cor}[thm]{Corollary}

\theoremstyle{definition}

\newtheorem{dfn}[thm]{Definition}

\theoremstyle{remark}
\newtheorem{rem}[thm]{Remark}
\newtheorem*{rmk}{Remark}

\numberwithin{equation}{section}

\def\tilde{\widetilde}

\renewcommand{\tilde}{\widetilde}

\def\R{\bb R}
\def\Z{\bb Z}
\def\CC{\bb C}
\def\C{\bb C}

\def\E{\bb E}
\def\EE{\bb E}
\def\N{\bb N}
\def\P{\bb P}

\def\T{\bb T}

\def\hat{\widehat}
\def\n{\nolimits}

\def\d{\delta}
\begin{document}

 \title{
Spectral gap properties of the unitary groups:\\
around Rider's results on non-commutative Sidon sets.}

\author{by\\
Gilles  Pisier\\
Texas A\&M University and UPMC-Paris VI}

\maketitle
\begin{abstract}
We present a proof of Rider's unpublished result that the union of two Sidon sets in the dual of a non-commutative compact group is Sidon, and that randomly Sidon sets are Sidon. Most likely this proof is essentially the one announced by Rider and communicated in a letter to the author around 1979 (lost by him since then). The key fact is a spectral gap property with respect to certain representations of the unitary groups $U(n)$ that holds uniformly over $n$. The proof crucially uses Weyl's character formulae. We survey the results that we obtained 30 years ago using Rider's unpublished results. Using a recent different approach valid for certain orthonormal systems of matrix valued functions, we give a new proof of the spectral gap property that is required to show that the union of two Sidon sets is Sidon. The latter proof yields a rather good quantitative estimate. Several related results are discussed with possible applications to random matrix theory.
 \end{abstract}
 
 
 \def\tr{{\rm tr}}

\vfill\eject

\tableofcontents

\vfill\eject
 
A subset $\Lambda$  of a discrete Abelian group  $\hat G$
is called Sidon if every continuous function on $G$ with Fourier transform
supported in
$\Lambda$ has an absolutely convergent Fourier series.

The study of 
 Sidon sets in discrete Abelian groups was  actively developed
in the 1970's and  1980's, after
  Drury's remarkable proof of the stability of Sidon sets under finite unions (see \cite{LR}). 
  Rider  \cite{Ri} connected Sidon sets to random Fourier
series. 
This led the author to
a new characterization of Sidon sets as $\Lambda(p)$-sets
(in  Rudin's sense) with constants $O(\sqrt p)$
and eventually   to an arithmetic characterization of Sidon sets (see \cite{Pi,MaPi,Pis}).
Bourgain \cite{Bo} gave a different proof of this.
  The 2013 book \cite{GH} by Graham and Hare gives an account
  of this subject, updating  the  1975  one \cite{LR}  by Lopez and Ross. 
  See also \cite{LQ} for connections with Banach space theory.
   
   Throughout this, the main example always remains the integers
   $\hat G=\Z$
  (with $G=\T=\R/\Z$), and Sidon sets are defined by the properties
   of Fourier series on $\T$ with coefficients
supported in the set.
   The classical example of a Sidon set is
   a set     formed of a sequence $\{n(k)\}$
   such that $\inf n(k+1)/n(k)>1$ (such sets are called ``Hadamard lacunary").
   While the theory was initially inspired by this first example,
   much of it rests on another one, where $\T$ is replaced by $G=\T^\N$
   (or by $\{-1,1\}^\N$),
   and the fundamental Sidon set in its dual $\hat G$
   is the one formed by the coordinate functions on $G$. In particular,
   the connections with random Fourier series are closely related to this
   second example.
   
    Sidon sets are the analogue for discrete groups of 
  the so-called ``Helson sets"
  in continuous groups.
  The latter subject was  actively studied in the late 1960's  and 1970's notably
  by  Kahane and Varopoulos in Orsay,   K\"orner in Cambridge
  and many more (see \cite{Ka-,Ka,GMc}). Indeed,
   Sidon sets
  were then quite popular in harmonic analysis: in the Polish school following an old tradition (Banach, Kaczmarz, Steinhaus, Hartman,...), in the US after Hewitt and Ross,    
but also in the Italian (around Fig\`a-Talamanca) and 
    Australian schools (around Edwards and Gaudry).
    
   The harmonic analysis of thin sets was extended
   already in the late 1960's    to subsets of the dual ``object"
   $\hat G$ of any non-commutative
   compact group $G$, with Fourier series
   replaced by  the Peter-Weyl orthogonal development of functions on $G$.
   In this setting pioneering work was done by Fig\`a-Talamanca
   and Rider (\cite{FTR1,FTR2,FT}) 
   on generalized random Fourier series. There was initially
   a lot of excitement around the opening that non-commutative
   compact groups offered as a substitute for $\T$. However, the subject was given a cold shower when it was discovered (see \cite{Ri2,Ri3,Ce,Hut})
   that even for the simplest example  $G=SU(2)$
   infinite Sidon sets do not exist. Since finite Sidon sets 
   were considered trivial, this brought this whole direction to a full stop
   and probably gave  a bad reputation to Sidon sets in the duals of non-commutative
   compact groups. After that, many 
   in the next generation of researchers, in particular in the
   Polish school (Bo\.zejko, Pytlik, Szwarc,...) and
   the  Italian one (Fig\`a-Talamanca, Picardello...), turned to harmonic analysis on free groups
   (see e.g. \cite{FTN,FTP}). In this setting
   free sets, or ``almost free" sets, such as  the so-called Leinert sets (see e.g. \cite{Leh})
   or
   $L$-sets in the sense of
     \cite{Pajm}, can be viewed as analogous in some sense to Sidon sets
   in discrete non-commutative groups.

   This context probably explains why  Rider, when he published
   in \cite{Ri} his theorem connecting Sidon sets and random Fourier
   series decided not to include the details on
  the proof of the same result for  subsets of the duals of non-commutative
   compact groups. In the commutative case, full details
   could be   included without any special technical 
   difficulty because the key ingredient
   was a variant of Drury's 
interpolation trick (by then well known), invented to prove that the union of two
Sidon sets is Sidon,  and actually Rider's theorem
could be viewed as a generalization of Drury's union theorem. However, the extension of the latter to the non-commutative case was far from obvious
(see Remark \ref{diff}), and in fact it was still open until Rider's \cite{Ri}. Nevertheless, Rider chose
 to only announce  there  that he had settled it and promised to include the details, which involved a delicate estimate based on Weyl's character formula
for the unitary groups (see Theorem \ref{t1}), in a later publication,  but he never did.

In the late 1970's the author proved a series of results on Sidon sets
all based initially on Rider's breakthrough from \cite{Ri}.
It turned out that essentially all these results could be extended
for subsets of $\hat G$ when $G$ is a non-commutative compact group
\cite{MaPi,Pi2}. However, the latter extension required
the non-commutative unpublished version of Rider's \cite{Ri}.
At the author's request at the time, Rider kindly communicated to him
a detailed handwritten proof of his key result in the non-commutative  case.
Unfortunately, although a copy of this letter was kept for a long time,
it seems now to have been lost. Perhaps the successive moves
of the Jussieu Math. Inst. are an excuse, but the guilt is on the author.
   The more so since Daniel Rider passed away in 2008.
   
   The main goal of this paper is to present the details
   of a proof of Rider's Theorem for subsets of 
   $\hat G$ when $G$ is a general (a priori non-commutative) compact group.   
   Toward the end we give another proof, quite different, that we recently obtained
   in a more general framework not requiring any group structure.
   
   The main point of
   Rider's proof  is a spectral gap property of the family
   $\{U(n)\mid n\ge 1\}$ formed of \emph{all} the unitary groups.
   The property involves the embedding $U(n)\to U(2n)$
   obtained by adding 1's on the main diagonal, but
   the relevant estimate has to be uniform over $n$.
   We feel that this property is
   of independent interest, likely to find applications
   in random matrix theory, now that the latter field   has become
    part of the
   main stream (much more so now than 40 years ago !).
  
   This motivated us to  include the full
   details of (what most likely was) Rider's proof.
   We then describe in \S \ref{rrs} how Rider derived from
   his spectral gap result the stability of Sidon sets under finite unions
   and the fact the Sidon property is equivalent to a weaker one
   involving random Fourier series  that we name ``randomly Sidon".
   
   In \S \ref{gsr} we survey the non-commutative  results that we obtained
   in the 1980's using Rider's unpublished work.
   Actually we take special care
   and give detailed proofs because we detected some
   exagerated claims there (in \cite{Pi2}) that we no longer believe are true.
   See Remark \ref{err}.
 
 In \S \ref{bc}, we single out several natural inequalities for random unitaries,
 related to the classical ones of  Khintchine for random signs. We review what is known and discuss the problem of finding
  the best constants for these.

   We   seize this occasion to   try to revive a bit the whole subject
   of Sidon sets in duals of non-Abelian compact groups
   in the light of the recent surge of interest in random matrix theory
   and Voiculescu's free probability (see \cite{VDN}). Indeed, although 
   finite sets  $\Lambda \subset \hat G$ are a trivial example of Sidon set, in
   the non commutative setting   one is led to  consider
   sequences of 
   compact groups $(G_n)$ and 
   sequences of  subsets $\Lambda_n \subset \hat G_n$ with uniformly
   bounded Sidon constants. Then even if the cardinality 
   of the subsets $\Lambda_n$
   is uniformly bounded (and in fact even if it is equal to 1 !)
the notion is interesting. The simplest (and prototypical)
example of this situation with $|\Lambda_n |=1$ is the case
when $G_n=U(n)$ the group of unitary  $n\times n$-matrices,
and $\Lambda_n$ is the singleton formed of the 
irreducible representation (in short irrep) defining
$U(n)$ as acting on $\CC^n$. Sets of this kind and various generalizations
were tackled early on by Rider under the name ``local lacunary sets" (see \cite{Ri4}), but we suspect that this setting of sequences of groups,
with uniform estimates,
which is nowadays commonly accepted, 
was viewed as not so natural at the time.

We illustrate  this   in Theorem \ref{t4}. There we consider   a sequence of compact groups $G_n$
   and a sequence of unitary irreps $\pi_n \in \hat G_n$ with 
   unbounded dimensions,
   and we focus on the situation when the singletons
   $\{\pi_n\}$ have uniformly bounded Sidon constants. We give
   several equivalent characterizations of this situation, in terms
   of the character $t\mapsto \tr(\pi_n(t))$ of $\pi_n$.
   Surprisingly, 
   this becomes void if one uses a sequence of finite groups,
   or of groups that are amenable as discrete groups.
   In that case the dimensions must remain bounded. 
    E. Breuillard   opened our eyes to this
   phenomenon. We refer the reader to the forthcoming
   paper \cite{BreP} for more on this.

\section{Notation. Background. Spectral gaps}

Throughout this section, let $G$ be a compact group. 
We denote by $\hat G$ the dual object
formed as usual of all the (equivalence classes of)  irreducible representations (irreps in short) on $G$. We identify two irreps when they are unitarily equivalent.
We denote by $M(G)$ the space of Radon measures on $G$
equipped as usual with the total variation norm $\mu\mapsto \|\mu\|_{M(G)}= |\mu|(G)$.
\def\p{\cl P}
\def\g{\gamma}

We denote by $M_d$ the space of all complex matrices
of size $d\times d$ with the usual operator norm
as acting on $\ell_2^d$.

We denote by $U(d)\subset M_d$ the compact group formed of all unitary matrices
of size $d\times d$.

For any measure $\mu$ on   $G$ and any irrep $\pi:\ G \to U(d_\pi)$
 we define the Fourier transform  
 by \begin{equation}\label{phr1}\hat \mu (\pi)=\int \ovl{\pi(t)}  \mu (dt) \in M_{d_\pi}.\end{equation}
Note that $\forall \mu_1,\mu_2\in M(G)$
\begin{equation}\label{phr}\hat {\mu_1\ast \mu_2}(\rho)=\hat \mu_1(\rho)\hat \mu_2(\rho).\end{equation}

 We denote by  $m_G$ the normalized Haar measure 
 and by $t_G\in \hat G$ the trivial representation on $G$.

We denote   $L_p(G)=L_p(G,m_G)$. We view
$L_1(G)$ as isometrically embedded in 
${M(G)}$ via $f\mapsto f m_G$. In particular,
the Fourier transform of any $f\in L_1(G)$ is defined as
\begin{equation}\label{phr2}\hat f (\pi)=\int \ovl{\pi(t)}  f(t) m_G(dt).\end{equation}
For any $f\in L_2(G)$ we have (Parseval)
$$\|f\|_2=(\sum\nolimits_{\rho\in \hat G} d_\rho \tr|\hat f(\rho) |^2)^{1/2},$$
and the Fourier expansion of $f$ takes the form
$$f=\sum\nolimits_{\rho\in \hat G}  d_\rho \tr( {}^t\hat f(\rho) \rho) .$$
\begin{rmk} Note that our definitions of $\hat \mu$ and $\hat f$
in \eqref{phr1} and \eqref{phr2}  differ from that
of \cite{HR}, where $\hat \mu(\pi)$ is defined as   $\int \pi(t)^*\mu(dt)$ and similarly
for $\hat f$. Thus the Fourier coefficient in the sense of \cite{HR}
is the transpose of what it is in our sense. The advantage
is that we have \eqref{phr} while the convention of \cite{HR}
requires to reverse the order
of the factors on the right hand side of \eqref{phr}.
 \end{rmk} 
We denote by $\chi_\pi$ the character of $\pi$, i.e. we have
$\chi_\pi(x)=\tr(\pi(x))$ for any $x\in G$.\\
 A   measure
$\mu\in M(G)$ (resp. a function $f\in L_1(G)$) is called \emph{central}
if $$\forall g\in G\quad
\mu =\delta_g \ast \mu\ast \delta_{g^{-1}}
 $$
 (resp.  $f =\delta_g \ast f\ast \delta_{g^{-1}}$).
  Then the Fourier transform  
   $\hat \mu  $ (resp. $\hat f  $) is scalar valued, i.e. 
   $\hat \mu (\pi)$
   or $\hat f (\pi)$
    belong to  the space of 
   scalar multiples of the identity matrix of size $d_\pi$.\\
   Thus the subspace of
   the central functions in $L_p$ ($1\le p<\infty$) coincides with 
   the closed linear span of the characters $\{\chi_\pi\mid \pi\in \hat G\} $.
   
   There is a  bounded linear projection  $P$
   from $M(G)$ onto the subspace of 
   all central measures, defined simply by
    \begin{equation}\label{30} P(\mu)= \int \delta_g \ast f\ast \delta_{g^{-1}} m_G(dg).\end{equation}
Clearly $\|  P(\mu)\|\le \| \mu \|$.
  We denote by $A(G)$ the Banach space 
formed of those $f :\ G\to \CC$ such that
$\sum_{\pi\in \hat G} d_\pi \tr|\hat f(\pi)|<\infty$,
and we equip it with the norm
$$\|f\|_{A(G)}= \sum\nolimits_{\pi\in \hat G} d_\pi \tr|\hat f(\pi)|.$$
\begin{dfn}[Sidon sets] A subset $\Lambda\subset \hat G$
is called Sidon if there is a constant $C$ such that
$$\|f\|_{A(G)}\le C \|f\|_{C(G)}$$
for any $f\in C(G) $  with Fourier transform supported in $\Lambda$.
More explicitly, this means that for any finitely supported family
$(a_\pi)$ with $a_\pi\in M_{d_\pi} $ ($\pi\in \Lambda$) we have
$$\sum\nolimits_{\pi \in \Lambda} d_\pi \tr |a_\pi| \le C \|\sum\nolimits_{\pi \in \Lambda} d_\pi \tr(\pi a_\pi)\|_\infty.$$
\end{dfn}

For any pair $f,h\in L_2(G)$, the convolution $f\ast h$ belongs to $A(G)$
and  \begin{equation}\label{20}\|
f\ast h\|_{A(G)}\le \|
f \|_{L_2(G)} \|
h \|_{L_2(G)}.\end{equation} 
Moreover, we have for any $f\in A(G)$ and any $\nu \in M(G)$
\begin{equation}\label{25}
 \int f d\nu   =
 \sum\nolimits_{\pi\in \hat G} d_\pi\tr ( {}^t \hat f(\pi) \hat\nu(\bar\pi) )
 = \sum\nolimits_{\pi\in \hat G} d_\pi \sum\nolimits_{i,j\le d_\pi }  {  \hat f(\pi)}_{ij} {\hat\nu(\bar\pi)}_{ij} .
\end{equation}
and hence
 \begin{equation}\label{21}
|\int f(g ) \nu(dg)|\le  
  \|f \|_{A(G)} 
\| \sup\nolimits_{\pi\in \hat G} \| \hat \nu(\pi) \|
.\end{equation}
More generally,
  let  $f,h\in L_\infty(G; M_d)$ ($d\ge 1$). We define the convolution
  $F=f\ast h$ using the matrix product in $M_d$, so that
  $F_{ij}=\sum_k f_{ik}\ast h_{kj}$. 
  Let $x,y$ be in the unit ball of $\ell_2^d$.
  We have then
   \begin{equation}\label{22}
    \|\langle F x,y\rangle\|_{A(G)}\le \| f\|_{L_\infty(G; M_d)} \|h\|_{L_\infty(G; M_d)}
   .\end{equation}
Indeed, this follows easily from (here we use \eqref{20})
$$\|\langle F x,y\rangle\|_{A(G)}\le \sum\nolimits_k \| \sum\nolimits_i \bar x_i f_{ik}  \|_2
\| \sum\nolimits_j  y_j h_{kj}  \|_2
\le ( \sum\nolimits_k \| \sum\nolimits_i \bar x_i f_{ik}  \|_2^2)^{1/2} ( \sum\nolimits_k \|  \sum\nolimits_j  y_j h_{kj} \|_2^2)^{1/2}$$
$$   = \left( \int\sum\nolimits_k | \sum\nolimits_i \bar x_i f_{ik}  |_2^2dm_G\right)^{1/2} \left( \int\sum\nolimits_k |  \sum\nolimits_j  y_j h_{kj} |_2^2 dm_G\right)^{1/2}\le \| f\|_{L_2(G; M_d)} \|h\|_{L_2(G; M_d)}
.$$
A fortiori, we obtain by \eqref{21}
 \begin{equation}\label{23}
    |\int \langle F(g) x,y\rangle  \nu(dg)| \le \| f\|_{L_\infty(G; M_d)} \|h\|_{L_\infty(G; M_d)}  \sup\nolimits_{\pi\in \hat G} \| \hat \nu(\pi) \|
   .\end{equation}
Taking the sup over $x,y$, we find
 \begin{equation}\label{24}
    \|\int   F(g)  \nu(dg)\|_{M_d} \le \| f\|_{L_\infty(G; M_d)} \|h\|_{L_\infty(G; M_d)}  \sup\nolimits_{\pi\in \hat G} \| \hat \nu(\pi) \|
   .\end{equation}

\noindent{\bf Notation:} Let ${\cl G}=\prod_{\pi\in {\hat G}} U(d_\pi)$.
Let $u\mapsto u_\pi \in U(d_\pi)$ denote the coordinates on $ \cl G$.
\begin{dfn}[Randomly Sidon] A subset $\Lambda\subset \hat G$
is called randomly Sidon if there is a constant $C$ such that
 for any finitely supported family
$(a_\pi)$ with $a_\pi\in M_{d_\pi} $ ($\pi\in \Lambda$) we have
$$\sum\nolimits_{\pi \in \Lambda} d_\pi \tr |a_\pi| \le C \int \|\sum\nolimits_{\pi \in \Lambda} d_\pi \tr({u}_\pi \pi a_\pi)\|_\infty m_{{\cl G}}(d{u}).$$
\end{dfn}
Note that in Lemma \ref{68} we 
give a simple general argument showing that replacing the random unitaries 
$(u_\pi )$ by standard complex Gaussian random matrices
(with the usual normalization) leads to  the same notion of ``randomly Sidon".

Clearly Sidon implies randomly Sidon (with the same constant).

We denote by $\p(G)\subset M(G)$ the set of probability measures on $G$.\\
We say that  $\Lambda\subset \hat G$ is symmetric if 
$\bar \pi\in \Lambda$ for any $\pi\in \Lambda$.
\begin{dfn}[Spectral gap]\label{d1}
Let $0\le \gamma<\d\le 1$. We will say that a probability measure $\mu\in \p(G)$ has
a $(\d,\g)$-spectral gap with respect to a symmetric subset $\Lambda\subset \hat G$
if $\hat\mu (\pi)=  \d I$ for any $\pi\in \Lambda $ and
$  \|\hat\mu(\rho)\|\le \g$ for any nontrivial ${\rho \not\in \Lambda  }$.
\end{dfn}
\begin{rem}[Spectral gap as an inequality]
 Let $E\subset L_2(G)$  be the subspace formed
 of those $f\in  L_2(G)$ such that  
  $\hat f(\pi)=0$ for any non-trivial $\pi\not \in \Lambda$. Let
$P:\ L_2(G)\to E$ denote the orthogonal projection.
Note $Pf=\int fdm_G+ \sum\n_{\pi\in \Lambda} d_\pi \tr({}^t\hat f(\pi) \pi)$ for any $f\in L_2(G)$.
Let $P_\d f = \int fdm_G+ \d\sum\n_{\pi\in \Lambda} d_\pi \tr({}^t\hat f(\pi) \pi)$.
Then $\mu $ has
a $(\d,\g)$-spectral gap with respect to   $\Lambda $ iff
$$\forall f\in L_2(G)\quad\| \mu \ast f -P_\d f  \|_2 \le  \g \|f-Pf\|_2.$$

\end{rem}
\begin{dfn}[$(\d,\g)$-isolated]\label{d1bis}
We will say that $\Lambda\subset \hat G$ is
$(\d,\g)$-isolated if there is $\mu\in \p(G)$ that has 
a $(\d,\g)$-spectral gap with respect to  $\Lambda $.
\end{dfn}
\begin{rem} Using the central projection \eqref{30}
we may always assume in the preceding that $\mu$ is a central measure.
\end{rem}

The basic example is the set $\Lambda=\{-1,1\}\subset \Z$. The measure
$\mu=(1+\cos(t)) m_{\T}(dt)$ has a $(1/2,0)$-spectral gap with respect to  $\Lambda $.\\
On $G=\{-1,1\}$ the measure $\mu =(1+\xi) m_G$ does the same with respect to
the set formed of the character   $\xi\in \hat G$ associated to the identity   map.\\
More generally, Riesz products  give more sophisticated examples.
Let $G$ be  a compact Abelian group. Let
$\{\gamma_n\mid n\in \N\}\subset \hat G$ be ``quasi-independent", i.e. such that
 there is no nontrivial choice of $(\xi_n)\in \{-1,0,1\}^\N$ finitely supported such that
 $\prod {\gamma_n}^{\xi_n}=1$. 
 Assume $-1\le \d_n \le 1$.
 Then
 the probability measures $\nu_k= \prod_{n\le k} (1+\d_n \Re(\gamma_n)) m_G$
converge weakly when $k\to \infty$ to a probability $\nu$ on $G$.
We refer to $\nu$ as the Riesz product associated to 
$ \prod  (1+\d_n \Re(\gamma_n)) $.

If we assume that $\d_n=\d$ for all $n$ and $0<\d<1$, then 
the Riesz product $\nu$ has a $(\d,\d^2)$-spectral gap with respect to  $\Lambda= \{\gamma_n\}\cup\{ \bar \gamma_n\}$. 
For instance, this holds
for $G=\R/2\pi\Z$ when $\Lambda= \{\gamma_n\}$ 
is identified to the subset  $ \{2^n\}\subset \Z$
by  $\gamma_n(t)=\exp{(i 2^n t)}$.
This also holds for $G=\{-1,1\}^\N$ 
(resp. $G=\T^\N$ ) 
when $\Lambda\subset \hat G$ is the set $\{\xi_n\}$
(resp. $\{\xi_n\}\cup \{\bar \xi_n\}$ ) with $(\xi_n)$ denoting the coordinates on $G$.

Let $\sigma_n:\ U(n)\to M_n$ be the ``defining" irrep,  i.e. the identity map on 
$U(n)$.

\begin{lem}\label{l12} Let $n\ge 1$.   
For any $0<\d\le 1/(2n)$,
let $$\varphi^\d_{n}= 1+\d (\chi_{\sigma_n}+\ovl{ \chi_{\sigma_n} }) =1+\d  ({\tr(\sigma_n)+\ovl{\tr(\sigma_n)}} ).$$
Let $\nu^\d_{n} \in M(U(n))$ be the probability measure
defined by
$\nu^\d_{n} = \varphi^\d_{n} \ m_{U(n)}.$
Then $\nu^\d_{n}$ has a $(\d/n,0)$-spectral gap with respect to $ \{ \sigma_n,\ovl{\sigma_n} \}.$
\end{lem}
\begin{proof} Obviously $\hat{ \varphi^\d_{n}}(\sigma_n)= \hat{ \varphi^\d_{n}}(\ovl{\sigma_n}) =\d/n$ and $\hat{ \varphi^\d_{n}}(\pi)=0$ for any other nontrivial irrep $\pi$.
\end{proof}

\begin{dfn}[peak sets]\label{d2}
 Let $0<\vp<1$.
We say that  $\Lambda\subset \hat G$ is an  $\vp$-peak set
with constant $w$
if there is $\nu\in M(G)$ with $\|\nu\|_{M(G)}\le w$ such that
 $\hat\nu (\pi)=   I$ for any $\pi\in \Lambda$ and
$\sup\nolimits_{\rho \not\in \Lambda} \|\hat\nu(\rho)\|\le \vp$.
\end{dfn}
\begin{rem}\label{r10}
If $\nu$ is as in Definition \ref{d2} for some  $0<\vp<1$
then $\nu^{\ast k}$ satisfies the same with $\vp^k,w^k$ in place of
$ \vp,w$. Therefore, if $\Lambda$ is an $\vp$-peak set for some 
$0<\vp<1$, then it is so for all $0<\vp<1$.
\end{rem}
\begin{dfn}[peaking Sidon sets]\label{d33}  
We say that a Sidon set $\Lambda\subset \hat G$ is   peaking if
 for any  $0<\vp<1$ and any ${u}\in \cl G$
 (or merely for any $u\in \prod\n_{\pi\in \Lambda} U(d_\pi)$) 
  there is a measure $\mu_\vp^{u}\in M(G)$
  such that 
  $$ \hat {\mu_\vp^{u}}(\pi)={u}_\pi \ \forall \pi \in \Lambda,\quad \sup_{\pi\not\in \Lambda}\|\hat   {\mu_\vp^{u}}(\pi)\|\le \vp  \text{  and }
  \|\hat {\mu_\vp^{u}} \|\le w(\vp)$$
  where $w(\vp)$ depends only on $\vp$.
 \end{dfn}
\begin{rem}[The main difficulty of the non-Abelian case]\label{diff} Note that one of our main goals will be to prove that actually
any Sidon set is peaking. This will be reached
in Theorem \ref{t2} and Remark \ref{phr5}.
Once this goal is attained, it follows as an easy  
corollary that the union of two Sidon sets is also one (see Corollary \ref{c2}).
In the Abelian case, Drury's (or Rider's) proof
made crucial use of the 
Riesz product $\prod (1+\d  (z_n+\bar z_n)/2 )$
on $\T^{\N}$ ($0\le \d<1$). With the notation in Lemma \ref{l12}
this is the same 
as  the infinite product of the probability  
$\nu^\d_{1}$ on $\T$. The latter has a $(\d,\d^2)$-spectral gap
with respect to the Sidon set  formed of the coordinates on $\T^{\N}$, which is the
fundamental example in the Abelian case. 
The proof that Sidon sets are peaking uses a certain
transplantation trick due to Drury to pass from the fundamental example to the general case. It is not really difficult to adapt that trick
to the non-Abelian case (see the proof of Theorem  \ref{t2}).
However, in the non-Abelian case the fundamental example
is the product $\prod_{n\ge 1} U(n)$ but the product
of the probabilities $\nu^\d_{n}$ \emph{fails} to have the required spectral gap, whence
  the need for a substitute for the Riesz product. This is precisely
 the role of  Theorem \ref{t1} in the next section.
\end{rem}
The preceding definitions are connected by the following simple result.
\begin{pro}\label{p1}  Let $0<\gamma<\d<1$.
Any $(\d,\g)$-isolated symmetric set $\Lambda\subset \hat G$ is an $\vp$-peak set
with constant $w$ for some $0<\vp<1$ and $w\ge 0$ depending only on
$\gamma,\d$.\\
Any Sidon set $\Lambda\subset \hat G$ that is also an $\vp$-peak set
with constant $w$ for some $0<\vp<1$ and $w\ge 0$ is peaking.
\end{pro}
\begin{proof} Let $\mu$ be as in Definition \ref{d1}. Let
$\nu= \d^{-1} (\mu-m_G)$ with $\vp=\g/\d$ and $w=d^{-1} (\|\mu\|+1)$. Then 
 $\nu$ satisfies the property in Definition \ref{d2}.
 If $\Lambda$ is Sidon with constant $C$, by (i) in Lemma \ref{l3} (Hahn-Banach), for
 any ${u}\in \cl G$
  there is a measure $\mu^{u}\in M(G)$
  such that 
  $$ \hat {\mu^{u}}(\pi)={u}_\pi \ \forall \pi \in \Lambda   \text{  and }
  \|\hat {\mu ^{u}} \|\le C.$$
  Let $\nu$ be as in Definition \ref{d2}. Then
   ${\mu_\vp^{u}}=  {\mu^{u}}\ast \nu$ is as in Definition \ref{d33}
   with $w(\vp)=Cw$.
   This gives the announced result for  some $0<\vp<1$,
   but replacing $\nu$ by its convolution powers we obtain a similar
   result for any  $0<\vp<1$.
   \end{proof}
   \begin{pro}\label{p2} 
 Let $G=\prod_{n\in \N} G_n$ be the product of a sequence of compact groups,
 let $(\mu_n)$ be a sequence with $ \mu_n\in \p(G_n)$ and
 let $(\Lambda_n) $ be a sequence of symmetric subsets with
  $\Lambda_n\subset \hat{G_n}$ for each $n$.
  Let $0<\gamma<\d<1$. Let $\gamma'=\max\{\gamma,\d^2\}<\d$.
  If $\mu_n$ has
a $(\d,\g)$-spectral gap with respect to $\Lambda_n $ for each $n$,
then the product $\mu=\otimes_{n\in \N} \mu_n$
has a $(\d,\g')$-spectral gap with respect to the subset
$\Lambda\subset \hat{G}$, 
denoted by $\dot\Sigma \Lambda_n$, 
consisting  of all the irreps $\pi$ on  $G$
of the following form:
  for some $n$ there is $\pi_n\in \Lambda_n$ such that
$$\forall x= (x_n)\in G\quad \pi(x)= \pi_n(x_n).$$
 \end{pro}

\begin{proof} Let $\pi \in \dot\Sigma \Lambda_n$.
Then $\hat \mu (\pi)= \hat {\mu_n} (\pi_n)$. Any nontrivial $\pi\in \hat G$
is of the form $\pi(x)=\otimes_{n\in \N} \pi_n(x_n)$ for
some sequence $(\pi_n)$ with $\pi_n\in \hat{G_n}$ containing some but only finitely many nontrivial terms.
If   at least one of these non trivial terms $\pi_n$
is not in $ \Lambda_n$, then $\| \hat \mu (\pi)\|\le \g$.
If they are all in $ \Lambda_n$ and $\pi \not\in \Lambda$,
there must be at least two of them and then $\| \hat \mu (\pi)\|\le \d^2$.
The result is then immediate.
\end{proof}
\begin{rem}\label{rl12} Let $G_k=U(d_k)$ and $G=\prod G_k$. Assume
that $N=\sup_k d_k <\infty$. Let $0<\d\le 1/(2N)$.
Let $\varphi_n\in L_1(G)$ be defined for $x=(x_k)\in G$ by
$\varphi_n(x)=\prod_{k\le n} (1+\d(\tr ( {x_k} ) +\tr (\ovl{ { x_k}} ))$, and let
  $\nu_n=\varphi_n m_G$.
As for Riesz products, $\nu_n\in {\cl P}(G)$, $\nu_n$ converges weakly to
some $\nu\in {\cl P}(G)$, and it is easy to check, similarly,
that $\nu$ has a $(\d/N, \d^2/N^2 )$-spectral gap. This can also be seen
as a particular case of the preceding Proposition with $\gamma=0$
and $\d$ replaced by $\d/N$.

\end{rem}
\section{The unitary groups}

The main difficulty 
Rider had to overcome to establish his main result is  the following
  spectral gap (and interpolation)  property of the sequence of the unitary groups
$\{U(n)\mid n\ge 1\}$ , which in our opinion,
is quite deep. Note however that, for the applications to Sidon sets,
any probability with the same gap property as the one denoted  below
by $\nu_n$ would do (see \S \ref{new}).
  
  Let $1\le k\le n$.
 Let $\Gamma(k)\subset U(n)$ be the copy of $U(k)$
  embedded in $U(n)$
      via $a\mapsto a\oplus I$.
Let $\mu_{k,n}$ be the central symmetric probability measure
defined by  
   \begin{equation}\label{w10} \mu_{k,n}=\int \delta_s \ast  m_{\Gamma(k)}\ast  \delta_{s^{-1}}\  m_{U(n)}(ds).\end{equation}
   
 We denote by $\sigma_n\in \widehat{{U(n)}}$    the defining representation of $U(n)$.
 
 We denote by $S_n\subset \widehat {U(n)}$ the set 
 $$S_n=\{ \sigma_n,\ovl{\sigma_n} \}.$$
 For  emphasis : it is crucial in the next statement that $\gamma  <1/2$ be \emph{independent of $n$}.
 
\begin{thm}\label{t1}[Rider, circa 1975, unpublished]\\
For any even $n\ge 2$, let $k=n/2$ and
let $\nu_n=  \mu_{k,n}$. \\
 For any odd $n $,  let $k_+=n/2+1/2$, $k_-=n/2-1/2$ and $\nu_{n}=1/2( \mu_{k_-,n}  +\mu_{k_+,n}   )$.\\
There is a positive constant $\gamma  <1/2$ such that for any $n\ge 4$,  the  symmetric central probability measure $\nu_n$
has a $(1/2,\g)$-spectral gap
with respect to $S_n$.
 More precisely, for any $1/4<\g<1/2$ this holds
 for all sufficiently   large  $n$. 
\end{thm}  
   
  \begin{rem} The case $n=1$, $G=\T=\R/2\pi\Z$ is classical.
  Then the probability measure $$\mu(dt)=(1+\cos t) m_\T(dt)$$
  (which is the building block for Riesz products)
  satisfies the analogous interpolation property, with $\gamma=0$.
  \end{rem}  
   \begin{cor}\label{c1} Let $(d_k)_{k\in I}$ be an arbitrary collection of integers.
   Let $G=\prod_{k\in I} U(d_k)$.
 Let $S\subset \hat G$ be the subset formed of all representations
 $\pi$ that, for some $k\in I$, are 
 of the form    $\pi(g)= \sigma_{d_k}(g_{k})$   ($g\in G$).
  For any $0<\vp<1$ there is a measure $\mu_\vp\in M(G)$
  such that
  $$ \hat \mu_\vp(\pi)=I \ \forall \pi \in S,\quad \sup_{\pi\not\in S}\|\hat \mu_\vp(\pi)\|\le \vp  \text{  and }
  \|\mu\|\le w(\vp)$$
  where $w(\vp)$ depends only on $\vp$.
  \end{cor} 
  \begin{proof} By Theorem \ref{t1}, there is
  $N$ (e.g. $N=4$) and $0<\g<1/2$ such that
  $S_n$ has a $(1/2,\g)$-spectral gap for any $n\ge N$.
    Let $G=G_1\times G_2$ with
  $G_1=\prod_{d_k< N} U(d_k)$ and $G_2=\prod_{d_k\ge N} U(d_k)$.
  Let $S_1\subset \hat{G_1}$ and $S_2\subset \hat{G_2}$ 
  be the corresponding subsets 
  and let $\Lambda_j=
  S_j\cup \ovl{S_j}$ ($j=1,2$).
  By Remark \ref{rl12}, 
  $\Lambda_1$ is $(\lambda,\lambda^2)$-isolated for any $0<\lambda\le 1/2N$. We may clearly assume $\g\ge 1/4$. Then
  by   Proposition \ref{p2}, 
  $\Lambda_2$ is $(1/2,\g)$-isolated.
  Taking convolution powers, we see that it is also
  $(1/2^m,\g^m)$-isolated for any integer $m\ge 1$.
  Choose $m$ minimal but large enough so that $1/2^m\le 1/(2N)$.
   Let $\d=1/2^m$ and
       $\g'=\max\{\g^m,\d^2\}$. Then
  both $\Lambda_1$ and $\Lambda_2$ are $(\d,\g')$-isolated.
  Therefore,   by   Proposition \ref{p2}
    $S\cup \bar S$ is also $(\d,\g')$-isolated.
   By Proposition     \ref{p1},
  $S\cup \bar S$ is an $\vp$-peak set for some $0<\vp<1$.
  Let $\nu_1\in M(G)$ be  such that $\hat \nu_1=I$ on $S\cup \bar S$ but
  $\|  \hat \nu_1\|\le \vp$ outside $S\cup \bar S$. It remains to
  show the same but with
  $S$  in place of $S\cup \bar S$.
  For any $z\in \T$, let $Z(z)\in G$ be the element
   such that $Z(z)_k=z I_k$. Note $\hat{\delta_{Z(z)} } (\sigma_k)=\bar z I_k$.
   Then let 
   $$\nu_2=  \int  z (\delta_{Z(z)} \ast \nu_1) m_{\T}(dz).$$
   Now $\hat \nu_2=I$ on $S $, and $\hat \nu_2=0$ on $\bar S $. 
   Also $\|\hat \nu_2\|\le \|\hat \nu_1\|$ on all of $\hat G$.
   Thus
  $\|  \hat \nu_1\|\le \vp$ outside $S$ and $\| \nu_2 \|\le \|\nu_1  \|$.
  By Remark \ref{r10} this completes the proof.
   \end{proof}

   We will need some background on irreps of the unitary groups. The ultraclassical reference is Hermann Weyl's \cite{W}.
   See e.g. \cite{Amr,Sag,Sta}  for more  recent accounts on the combinatorics of this rich subject.  We greatly benefitted
   from the expositions in \cite{Fa} and \cite{Fu}.
   
 Recall that for any compact group $G$, the set $\hat G$ consists of    
    irreps on $G$ with exactly one  representative, up to unitary equivalence, 
    of each   irrep. 
       Let $G=U(n)$. Then $\hat G$ is in 1-1 correspondence
       with the set of $n$-tuples $m=(m_1,m_2,\cdots,m_n)$ in $\Z^n$
       such that $m_1\ge \cdots\ge m_n$.
       Let $t=(t_1,\cdots,t_n) \in \CC^n$.
       Let $A_m(t )$ denote the determinant of the $n \times n$-matrix
       $a_m(t )$ defined by
       $$a_m(t )_{ij}= t_i^{m_j}.$$
       Let $\delta=(n-1,n-2,\cdots,1,0)$.
       Let $\pi_m$ be the irrep corresponding to $m$,
       and let $\chi_m$ denote its character.
       Then for any unitary $g\in U(n)$ with eigenvalues
       $t=(t_1,\cdots,t_n) \in \T^n$, $g$ is unitarily equivalent
       to the diagonal matrix $D(t)$ with coefficients $t$.
       This implies that $\chi_m(g)={\rm tr} (\pi_m(g))={\rm tr} (\pi_m(D(t)))= \chi_m(D(t))$. For simplicity,
       we will identify $t$ with  $D(t)$ and we set
       $\chi_m(t)=\chi_m(D(t))$. 
       We can now state Weyl's fundamental character formula,
       which goes back to \cite{W} : 
       \begin{equation}\label{w1}
       \chi_m(t)=\frac{A_{m+\delta}(t )}{A_{ \delta}(t )}.\end{equation}
       Note that $A_{ \delta}(t )$ is but the classical Vandermonde determinant
       $$A_{ \delta}(t )=\prod_{i<j} (t_i-t_j).$$
       We observe that for any $d\in \Z$ we have
       $$A_{m+(d,\cdots,d)}(t )=(t_1t_2\cdots t_n)^{d}  A_{m}(t )$$
       and hence for any $g\in G$
       $$\chi_{m+(d,\cdots,d)}(g )={\rm det}(g)^{d}  \chi_m(g).$$
       Thus if we choose $d=-m_n$, and set $\lambda_j=m_j+d$,
       we have $\lambda_1\ge \cdots \lambda_{n-1}\ge \lambda_n=0$,
       and
        \begin{equation}\label{w13}\chi_m(g)= {\rm det}(g)^{m_n} \chi_\lambda(g).\end{equation}
              
       \def\det{{\rm det}}
\begin{rem}\label{raa}[Distinguished representations of $U(n)$] The trivial representation of $U(n)$ corresponds to $m_1=\cdots=m_n=0$, so that $d=0$ and $\lambda_1=\cdots=\lambda_n=0$, and 
then $\chi_m(t)=1$ for all $t\in U(n)$.\\
The representation $\sigma_n(t)=t$ corresponds to
$m=\lambda=(1,0,\cdots,0)$ and $d=0$.
Then $$\chi_m(t)=t_1+\cdots+ t_n.$$
The representation $\sigma_n(t)=\bar t$ corresponds to
$m =(0,\cdots,0,-1)$ or equivalently to
$ \lambda=(1,\cdots,1,0)$ and $d=1$. Then $$\chi_m(t)=\bar t_1+\cdots+\bar t_n=(\prod\nolimits_{j\not=1} t_j+\cdots+\prod\nolimits_{j\not=n} t_j) \ \det(t)^{-1} .$$
In the sequel, we denote
$$\lambda_+=(1,0,\cdots,0)\quad\text{and}\quad
\lambda_-=(1,\cdots,1,0).$$
 \end{rem}
       The point of 
       \eqref{w13} is that
       now $\lambda$ can be identified with a Young diagram
       with a first row of $\lambda_1$ boxes,
       sitting as usual above a second row of $\lambda_2$ boxes, and so on.
        This will allow us to take advantage of the so-called Jacobi-Trudi
       formula (see \cite[p. 75]{Fu}) :
          \begin{equation}\label{w2}\chi_\lambda(t)=s_\lambda(t),\end{equation}
       where $s_\lambda$ 
       is the famous Schur symmetric polynomial in $t=(t_1,\cdots,t_n)$, which can be defined for $\lambda\not=0$ as the sum
       \begin{equation}\label{w2'}s_\lambda(t)=\sum t^T\end{equation}
      running over all the admissible fillings (or ``tableaux") $T$
      of the diagram $\lambda$ with the numbers $1,2,\cdots,n$.
      Here an admissible filling assigns
      to any box a number in $1,2,\cdots,n$
      so that the numbers are strictly increasing when running down
      a column and weakly increasing along each row, 
      and $$t^T=\prod\nolimits_{1\le i\le n}  t_i^{r_i}  $$
      where $r_i\ge 0$ is the number of times $i$ is used in the filling $T$.\\
      By convention, for the case $\lambda_1=\cdots=\lambda_n=0$, we set $s_0(t)=1$.
      
      Let  $1^n=(1,\cdots,1)$ where $1$ is repeated $n$ times.
      Then  \eqref{w2'} implies
      \begin{equation}\label{w2''} s_\lambda(1^n)=|\{ T\}|,\end{equation}
     i.e. $ s_\lambda(1^n)$ is the number of admissible fillings
     of $\lambda$ with the numbers $1,2,\cdots,n$.

      Then for any $\lambda=( \lambda_1,\cdots,\lambda_n )$
      with $\lambda_1\ge\cdots \ge\lambda_n\ge 0$ we have
        \begin{equation}\label{w3}\chi_\lambda(1^n)=s_\lambda(1^n)
        =\prod_{i<j} \frac{\lambda_i-\lambda_j+j-i}{j-i}.\end{equation}
         Note that
    $ \frac{\lambda_i-\lambda_j+j-i}{j-i} \ge 1$ for all $i<j$.

      This classical formula can be deduced from \eqref{w1}:
      by setting $t=(1,x,x^2,\cdots,x^{n-1})$, 
    and observing that $A_{\lambda+\delta}(1,x,x^2,\cdots,x^{n-1})$
      is a Vandermonde determinant, we have
      $$\chi_\lambda(1,x,x^2,\cdots,x^{n-1})
      =x^{\sum (i-1)\lambda_i} \prod_{i<j} \frac {x^{\lambda_i-\lambda_j +j-i} -1}{  x^{ j-i}-1}.
      $$ Then letting $x$ tend to $1$,
       and making the obvious
      common division in numerator and denominator,
      \eqref{w3} follows.

     The preceding definition of 
   the Schur symmetric polynomial $s_\lambda$
   is classically given as a function of  $k$-variables
   with $k$ not necessarily equal to the number of rows $n$ of  $\lambda$: one sets
    $$s_\lambda(t_1,\cdots,t_k)=\sum t^T$$
    where   the sum runs over all the admissible fillings
    of the Young diagram $\lambda$ by the numbers $1,2,\cdots,k$, with $t^T$ as before.
    
If  $\lambda_n>0$ and $k<n$, then the first column has length $>k$, so there are no admissible fillings by $(1,\cdots,k)$ and  $s_\lambda(t_1,\cdots,t_k)=0$ in that case.
      
      We now fix $1\le k<n$. \\
      We wish to compute the restriction of $\chi_\lambda$
      to the subgroup $U(k)$ viewed as embedded in $U(n)$
      via $a\mapsto a\oplus I$ or equivalently $a\mapsto\left(\begin{matrix} a &  0\\
       0& I_{n-k}\end{matrix}\right)$.
       In other words we are after a formula for
      $\chi_\lambda(t_1,\cdots,t_k,
      1^{n-k})$. We find it convenient to use
      \eqref{w2} and   \eqref{w2'}.
      Note that any admissible filling of $\lambda$ by $ ( 1,\cdots, n)$
      induces by restricting it to $ ( 1,\cdots, k)$ a filling
      of  a   diagram $\mu\le \lambda$, in the sense
      that $\mu_i\le \lambda_i$ for all $1\le i\le n$.
      The remaining set of boxes, denoted by
      $\lambda\setminus \mu$ is (in general) no longer a   diagram,
      it is only what is called a skew diagram, but the rule for filling
      it is respected by the induced numbering on its rows and columns,
      so that we can extend  to $\lambda\setminus \mu$  the notation
      \eqref{w2'}.
      Thus to any admissible filling 
of $\lambda$ by $ ( 1,\cdots, n)$ we associate
$\mu\le \lambda$ with a filling by $ ( 1,\cdots, k)$ 
and $\lambda\setminus \mu$ with a filling by $ ( k+1,\cdots, n)$.
Conversely, a moment of thought shows that 
separate admissible fillings of $\mu$ by $ ( 1,\cdots, k)$
      and $\lambda\setminus \mu$  by $ ( k+1,\cdots, n)$
   can be joined to form a filling of $\lambda$ by $ ( 1,\cdots, n)$.
   This leads to the identity (see \cite[p. 175]{Sag})
   \begin{equation}\label{w5} s_\lambda(t)=\sum_{\mu\le \lambda} s_\mu(t_1,\cdots,t_k)
   s_{\lambda\setminus \mu}(t_{k+1},\cdots,t_n)
,\end{equation} 
where again we set by convention  $s_{\lambda\setminus \mu}(t_{k+1},\cdots,t_n)=1$ if $\mu=\lambda$.\\
Moreover, we write $\mu \subset\lambda$ when $\mu_i\le \lambda_i$ for all $1\le i\le n$.

\begin{lem}\label{l10} Recall that  $\mu_{k,n}$ is the central symmetric probability measure
defined by  \eqref{w10}.
Let $m=(m_1,\cdots,m_n)\in \Z^n$, and let 
$\lambda_j= m_j-m_n$ ($1\le j\le n$). The Fourier transform
of   $\mu_{k,n}$ is as follows:
If $m_n>0$ we have
$\hat{\mu_{k,n}} (\pi_m)= 0$.\\
If $m_n\le 0$, let $d=-m_n$ and let $[d]^k=(d,\cdots,d,0,\cdots,0)$ with $d$ repeated $k$-times. Then
$\hat{\mu_{k,n}} (\pi_m)= 0$ unless $[d]^k\subset \lambda$ in which case we have
  \begin{equation}\label{e10}\hat{\mu_{k,n}} (\pi_m)= \frac{s_{\lambda\setminus [d]^k}(1^{n-k})}{s_{\lambda }(1^{n})}.\end{equation}
\end{lem}

\begin{proof}
We denote by $ (t_1,\cdots,t_k,1^{n-k})$ the eigenvalues of $g\in \Gamma(k)$,
with $t=(t_1,\cdots,t_k)\in \T^k$.
Then
$$\widehat{\mu_{k,n}}(\pi_m)=   \frac{1}{\dim(\pi_m)} F_{k,n}(\pi_m) \ I $$ where
by \eqref{w13}
$$F_{k,n}(\pi_m)= \int  \ovl{ \det(g)^{-d} \chi_\lambda(g)} m_{\Gamma(k)}(dg)= \int   {(t_1\cdots t_k)^{d}  } \ \ovl{ \chi_\lambda(t_1\cdots t_k,1^{n-k})} m_{\Gamma(k)}(dg).$$
By \eqref{w5} we have 
$$\chi_\lambda(t_1\cdots t_k,1^{n-k} )  
=\sum_{\mu\le \lambda} s_\mu(t_1,\cdots,t_k)
   s_{\lambda\setminus \mu}(1^{n-k}).$$
   Since the characters of $ {\Gamma(k)} $ are orthonormal in $L_2(m_{\Gamma(k)})$
  the integral
   $$\int {(t_1\cdots t_k)^{d}  } \ \ovl{ 
   s_\mu(t_1,\cdots,t_k)}
    m_{\Gamma(k)}(dg)  $$
    is  $=1$ if $\pi_\mu$ is equivalent to the irrep $g\mapsto \det(g)^d$ on $\Gamma(k)$, and $=0$ otherwise. \\ Since $g\mapsto \det(g)^d$ on $\Gamma(k)$
    corresponds to $(d,\cdots,d)$ ($k$-times) on $U(k)$, we have
   $$F_{k,n}(\pi)=\sum_{\mu\le \lambda} 
   \int {(t_1\cdots t_k)^{d}  } \ \ovl{ 
   s_\mu(t_1,\cdots,t_k)}
    m_{\Gamma(k)}(dg) \  s_{\lambda\setminus \mu}(1^{n-k})
      = s_{\lambda\setminus [d]^k}(1^{n-k}).
    $$
    More precisely,  $ \widehat{\mu_{k,n}}(\pi)=0$  for all $d<0$, and
    also 
   $ \widehat{\mu_{k,n}}(\pi)=0$ whenever $[d]^k\not \le \lambda$.
   Thus, if  $[d]^k \le \lambda$ and $0\le d\le \lambda_k$, we have
   $$F_{k,n}(\pi)= s_{\lambda\setminus [d]^k}(1^{n-k}).
    $$
    Moreover
     \begin{equation}\label{w12}
     {\dim(\pi_m)} = {\dim(\pi_\lambda)}= \chi_\lambda(1_G)=s_{\lambda }(1^{n}) .\end{equation}
This proves \eqref{e10}.
\end{proof}

\begin{lem}\label{l11}  Let $1\le k< n$. Let   $\lambda=( \lambda_1,\cdots,\lambda_n )$
      with $\lambda_1\ge\cdots \ge\lambda_n= 0$. 
       Assume $[d]^k\subset \lambda$ or equivalently
$0\le d\le \lambda_k$. Let $\lambda'=( \lambda_1,\cdots,\lambda_k )\setminus [d]^k$
and
$\lambda''=( \lambda_{k+1},\cdots,\lambda_{n} )$.
Then
$$s_{\lambda\setminus [d]^k}(1^{n-k})\le s_{ \lambda'}(1^{n-k})s_{ \lambda''}(1^{n-k}) .$$
 We have equality if   $ \lambda_{k+1}\le d$. \\
Moreover, $s_{\lambda\setminus [d]^k}(1^{n-k})=0$ 
if $d<\lambda_{n-k+1}$
(and a fortiori if $k+1> n-k$ and $d<\lambda_{k+1}$).
\end{lem}
\begin{proof} To any  admissible filling of
${\lambda\setminus [d]^k}$
we may associate, by restriction, an admissible filling of $\lambda'$ and one of $\lambda''$. Since
this correspondence is clearly injective, the inequality follows from \eqref{w2''}. 
Equality holds if it is surjective. Consider a pair of separate fillings
of $\lambda'$ and   $\lambda''$. 
If $ \lambda_{k+1}\le d$  there is no problem
   to join them into a filling of ${\lambda\setminus [d]^k}$,
so we have surjectivity.
If $ \lambda_{k+1}> d$ there may be an obstruction, however
$s_{\lambda\setminus [d]^k}(1^{n-k})=0$
if $d<\lambda_{n-k+1}$, because
one cannot fill the $(d+1)$-th column strictly increasingly by $1,\cdots,n-k$
(that column being of length $\ge n-k+1$ is too long for that). 
\end{proof}

\begin{lem}\label{l13} With the same notation as in  Lemma \ref{l11}:
\begin{itemize}
\item[(i)] If $d=0$ then $ \hat{\mu_{k,n}} (\pi_m)=0$ if $\lambda_{n-k+1}>0$, and
$$ \hat{\mu_{k,n}} (\pi_m)=   \left( \prod_{  i <j , \ j>n-k}   \frac{\lambda_i +j-i}{j-i}\right)^{-1}\text{   if   } \lambda_{n-k+1}=0.
$$
\item[(ii)] If $d\ge 1$,  $[d]^k\subset \lambda$ 
and $n-k\le k$ then
$$ \hat{\mu_{k,n}} (\pi_m)\le  \left( \prod_{  i\le k <j}   \frac{\lambda_i-\lambda_j+j-i}{j-i}\right)^{-1}.
$$
 
\item[(iii)] Moreover, $\hat{\mu_{k,n}} (\pi_m)=0$ if   $\lambda_k< d$ or if $d<\lambda_{n-k+1}$.
\end{itemize}
\end{lem}
\begin{proof} We will use \eqref{e10}. Recall $\lambda_n=0$. \\
(i) Assume $d=0$. Clearly, $s_{\lambda }(1^{n-k})=0$ if $\lambda_{n-k+1}>0$,
   because then we cannot fill the first column. \\ Now
    assume $\lambda_{n-k+1}=0 \ (=\lambda_n)$.
      By \eqref{w3} we have then
$ s_{\lambda }(1^{n-k}) =   \prod_{  i <j\le n-k}   \frac{\lambda_i-\lambda_j+j-i}{j-i} $, and hence
by \eqref{e10} and \eqref{w3} 
   $$  \hat{\mu_{k,n}} (\pi_m) =   \left( \prod_{  i <j , \ j>n-k}   \frac{\lambda_i +j-i}{j-i}\right)^{-1}.
   $$
  (ii) 
   Let $\mu=[d]^k$. Note that
   $[d]^k\subset \lambda$ implies  $\lambda_k\ge d$.
   With the notation of Lemma \ref{l11}, since by \eqref{w2''}
   $n-k\le k$ clearly implies $s_{ \lambda'}(1^{n-k})\le s_{ \lambda'}(1^{k})$,
   we have 
   $$s_{\lambda\setminus [d]^k}(1^{n-k})\le s_{ \lambda'}(1^{k}) s_{ \lambda''}(1^{n-k}) .$$
 We note that $\lambda'_i= \lambda_i-d$ for $ i\le k$
 and $\lambda''_i=\lambda_{k+i}$ for $i\le n-k$. Therefore, by \eqref{w3} 
 on one hand
  \begin{equation}\label{w8} s_{ \lambda'}(1^{k})=\prod_{i<j\le k}  \frac{\lambda'_i-\lambda'_j+j-i}{j-i} =
\prod_{i<j\le k}  \frac{\lambda_i-\lambda_j+j-i}{j-i} ,\end{equation}
and on the other hand
$$ s_{ \lambda''}(1^{n-k})=\prod_{i<j\le n-k}  \frac{\lambda''_i-\lambda''_j+j-i}{j-i} =
\prod_{i<j\le n-k}  \frac{\lambda_{k+i}-\lambda_{k+j}+{k+j}-{k+i}}{{k+j}-{k+i}}$$
or equivalently
  \begin{equation}\label{w9} s_{ \lambda''}(1^{n-k})=
 \prod_{k<i<j\le n }  \frac{\lambda_i-\lambda_j+j-i}{j-i}.\end{equation}
 Dividing  the product of 
  \eqref{w8} and   \eqref{w9} by
   $s_\lambda(1^{n})$ as given by   \eqref{w3},
  we obtain our claim  (ii). \\ \\
(iii) If $\lambda_k< d$, then $[d]^k\subset \lambda$ is impossible,
  and if $d<\lambda_{n-k+1}$ the  $(d+1)$-th column of $\lambda$
  has length $\ge n-k+1$ and hence cannot be filled
  strictly increasingly by $(1,\cdots,n-k)$, so that
 $    s_{\lambda\setminus [d]^k}(1^{n-k})=0$. Thus     $\hat{\mu_{k,n}} (\pi_m)=0$ by \eqref{e10}.
\end{proof}

\begin{lem}\label{l13b} With the same notation as in   Lemma \ref{l11}:
\begin{itemize}
\item[(i)] If $d=0$ then  
\begin{equation}\label{e11} \hat{\mu_{k,n}} (\pi_m)\le    \frac{(n-k)(n-k+1)}{n(n+1)} \text{   if   } \lambda_1\ge 2.
\end{equation}
\begin{equation}\label{e12} \hat{\mu_{k,n}} (\pi_m)\le    \frac {(n-k)(n-k-1)}{n(n-1)}  \text{   if   } \lambda_1=1 \text{ and } \lambda\not=\lambda_+.
\end{equation}
\item[(ii)] If $d\ge 1$ and $n-k\le k$  then
\begin{equation}\label{e13} \hat{\mu_{k,n}} (\pi_m)\le \frac{(n-k)(n-k+1)}{n(n+1)} \text{   if   } \lambda_k\ge 2 .\end{equation}
\begin{equation}\label{e14} \hat{\mu_{k,n}} (\pi_m)\le \frac {(n-k)(n-k-1)}{n(n-1)}  \text{   if   } \lambda_k=\lambda_1= 1 \text{   but  } \lambda\not=\lambda_- .\end{equation}
\begin{equation}\label{e15}  \hat{\mu_{k,n}} (\pi_m)\le  \frac {(n-k)(n-k-1)}{n(n-1)} 
\text{  if  } n-1>k
, \  \lambda_k= 1 ,  \lambda_1\ge 2 \text{   and   } \lambda_{n-1}=0.\end{equation}
\begin{equation}\label{e16}  \hat{\mu_{k,n}} (\pi_m)\le \frac {k(n-k)}{(n+1)(n-1)}  \text{   if   } \lambda_k= 1 , \lambda_1\ge 2  \text{   and   } \lambda_{n-1}\ge 1.\end{equation}
\end{itemize}
\end{lem}
\begin{proof} 
(i) We will use  Lemma \ref{l13} (i). Note that if $\lambda_i\ge \mu_i\ge 0$ for all 
   $i\le n$ we must have
   $$\left( \prod_{  i <j , \ j>n-k}   \frac{\lambda_i +j-i}{j-i}\right)^{-1}\le \left( \prod_{  i <j , \ j>n-k}   \frac{\mu_i +j-i}{j-i}\right)^{-1}.$$
   Assume first that $\lambda_1\ge 2$. We compare $\lambda$ with $\mu=(2,0,\cdots,0)$.
   Then
   $$\prod_{  i <j , \ j>n-k}   \frac{\mu_i +j-i}{j-i}\ge \prod_{   \ j>n-k}   \frac{\mu_1 +j-1}{j-1}= \frac{n(n+1)}{(n-k)(n-k+1)}.$$
   Now assume $\lambda_1=1$. Then $\lambda=(1,\cdots,1,0,0,\cdots)$
   where $1$ appears $r$-times.\\  If $\lambda\not=\lambda_+$ (see Remark \ref{raa})
   we must have $r\ge 2$. Then comparing $\lambda$ with $\mu=(1,1,0,\cdots,0)$, we obtain
  $$  \prod_{  i <j , \ j>n-k}   \frac{\lambda_i +j-i}{j-i} \ge   \prod_{   j>n-k}   \frac{\mu_1 +j-1}{j-1}  \prod_{    j>n-k}   \frac{\mu_2+j-2}{j-2}  = \frac{n}{n-k}  \frac{n-1}{n-k-1}.$$
 This proves (i).\\
  We now turn to (ii). Assume $d\ge 1$.
    We   use   Lemma \ref{l13} (ii) but 
    we  distinguish several subcases:
 
 $\dagger $ Assume first that $\lambda_k\ge 2$.
  Then, since $\lambda_n=0$
  $$ \prod_{  i\le k  }   \frac{\lambda_i-\lambda_n+n-i}{n-i}\ge \prod_{  i\le k  }
  \frac{2+n-i}{n-i}= \frac{n(n+1)}{(n-k)(n-k+1)}.$$
  This proves \eqref{e13}.
  
   $\dagger\dagger$ Now assume $\lambda_k=1$, so that $d=1$. Then the case $\lambda_1=1$
   is easy. Indeed, let $k\le s< n$ be such that
   $\lambda_j=1$ for $j\le s$ and $\lambda_j=0$ for $j> s$. 
   Since we exclude
   $\lambda_-$, we know that $s<n-1$ (see Remark \ref{raa})
, and hence
   $\lambda_{n-1}=0$. When $n=2$ this is impossible. When $n=3$,
   the only possibility is $k=1$ and then $\lambda\setminus[1]^k=0$,
   and hence $   \hat{\mu_{k,n}} (\pi_m) =0$.
   Therefore, we may restrict to $n\ge 4$. Note $s<n-1$ guarantees
   $k<n-1$.
   Then using both $j=n$ and $j=n-1$ we find
   $$\prod_{i\le k<j }  \frac{\lambda_i-\lambda_j+j-i}{j-i}\ge 
   \prod_{i\le k }  \frac{\lambda_i +n-i}{n-i}
   \prod_{i\le k  }  \frac{\lambda_i +n-1-i}{n-1-i}\ge
    \prod_{i\le k }  \frac{1 +n-i}{n-i}
   \prod_{i\le k  }  \frac{ n -i}{n-1-i}$$
   $$ = \frac{n(n-1)}{(n-k)(n-k-1)}. $$
   This proves \eqref{e14}.

   $\dagger\dagger\dagger$  Now assume $\lambda_k=1$ (and hence $d=1$) and   $\lambda_1\ge 2$. 
   \\ {\bf Case 1}. Assume first that $\lambda_{n-1}=0$.
   Then, assuming $n-1>k$ 
   
   $$\prod_{i\le k<j }  \frac{\lambda_i-\lambda_j+j-i}{j-i}\ge
   \prod_{i\le k, \ j\in\{n,n-1\} }  \frac{\lambda_i-\lambda_j+j-i}{j-i}\ge
   \prod_{i\le k } \frac{\lambda_i +n-i}{n-i}    \prod_{i\le k } \frac{\lambda_i +n-1-i}{n-2}$$
   $$\ge  \prod_{i\le k }    \frac{1 +n-i}{n-i} \prod_{i\le k }\frac{1 +n-1-i}{n-1-i}=\frac{n}{n-k}\frac{n-1}{n-k-1}.$$
  \\ {\bf Case 2}. Now assume $\lambda_{n-1}\ge 1$. Since we still assume
  $\lambda_k=1$ and   $\lambda_1\ge 2$,  we can compare $\lambda$ with
  $\mu$ defined by
   $\mu_1=2$, $\mu_i=1$ for all $i<n$ and $\mu_n=0$. Since $\lambda\ge \mu$ and
   $\mu_j=\lambda_j$ for all $j>k$
    we have  
  $$\prod_{i\le k<j }  \frac{\lambda_i-\lambda_j+j-i}{j-i}\ge
   \prod_{i\le k<j }  \frac{\mu_i-\mu_j+j-i}{j-i}= \prod_{  k<j<n }  \frac{\mu_1-\mu_j+j-1}{j-1} \prod_{i\le k}  \frac{\mu_i-\mu_n+n-i}{n-i}$$
but
$$   \prod_{  k<j<n }  \frac{\mu_1-\mu_j+j-1}{j-1}=  \frac{n-1}{k}
\text{   and   }
        \prod_{i\le k}  \frac{\mu_i-\mu_n+n-i}{n-i}=\frac{n+1}{n-1}
        \prod_{1<i\le k}  \frac{n-i+1}{n-i}=\frac{n+1}{n-k}         $$
        and hence
        $$\prod_{i\le k<j }  \frac{\mu_i-\mu_j+j-i}{j-i}\ge\frac{(n-1)(n+1)}{k(n-k)}
           .$$
          This proves \eqref{e15}.
\end{proof}
\begin{rem}\label{bra} In the proof of part (ii) in the preceding Lemma \ref{l13b}
the majorizations of  $\hat{\mu_{k,n}} (\pi_m)$ appearing there are all proved actually
for    $\left( \prod_{  i\le k <j}   \frac{\lambda_i-\lambda_j+j-i}{j-i}\right)^{-1}$.
\end{rem}
\begin{lem}\label{l14} With the same notation as in   Lemma \ref{l11}, let $n> 3$ be an odd integer  and let
$k=(n-1)/2>1$ so that $n=2k+1$.
 If $d\ge 1$    then
\begin{equation}\label{e13-} \hat{\mu_{k,n}} (\pi_m)\le \frac{(n-k)(n-k-1)}{n(n+1)} \text{   if   } \lambda_k\ge 2 .\end{equation}
\begin{equation}\label{e14-} \hat{\mu_{k,n}} (\pi_m)\le \frac {(n-k-1)(n-k-2)}{n(n-1)}  \text{   if   } \lambda_k= 1 \text{   and   }\lambda_1= 1 \text{   but  } \lambda\not=\lambda_- .\end{equation}
\begin{equation}\label{e15-}  \hat{\mu_{k,n}} (\pi_m)\le  \frac {(n-k-1)(n-k-2)}{n(n-1)} \text{   if   } \lambda_k= 1 ,  \lambda_1\ge 2 \text{   and   } \lambda_{n-1}=0.\end{equation}
\begin{equation}\label{e16-}  \hat{\mu_{k,n}} (\pi_m)\le \frac {(k+1)(n-k-1)}{(n+1)(n-1)}  \text{   if   } \lambda_k= 1 , \lambda_1\ge 2  \text{   and   } \lambda_{n-1}\ge 1.\end{equation}
\end{lem}

\begin{proof} We again decompose $\lambda$ into
$\lambda'$ and $\lambda''$, but we will modify the definition of $\lambda'$. Now
$\lambda'$ will have $k+1$ rows. Its first $k$ rows being as before
the same as those of $\lambda\setminus [d]^k$, and the $(k+1)$-th row
being like this: if $\lambda_{k+1}<d$ we set $\lambda'_{k+1}=0$, while
if $\lambda_{k+1}\ge d$ we set $\lambda'_{k+1}=\lambda_{k+1}-d$.
As for $\lambda''$ it is formed as before of the last $n-k$ rows of $\lambda$.
Then  arguing as in
 Lemma \ref{l11}  we find
$$s_{\lambda\setminus [d]^k}(1^{n-k})\le s_{ \lambda'}(1^{k+1})s_{ \lambda''}(1^{n-k}) .$$
By \eqref{e10}
we have $$
\hat{\mu_{k,n}} (\pi_m)\le  \frac{s_{ \lambda'}(1^{k+1})s_{ \lambda''}(1^{n-k})}{s_{\lambda }(1^{n})}.$$
We now use \eqref{w3} for $\lambda'$, $\lambda''$ and $\lambda$. This gives us
 $$
\hat{\mu_{k,n}} (\pi_m)\le 
\left( \prod_{  i\le    k+1 <j}   \frac{\lambda_i-\lambda_j+j-i}{j-i}\right)^{-1}
\prod_{  i\le k} \frac{\lambda'_i-\lambda'_{k+1}+k+1-i}{\lambda_i-\lambda_{k+1}+k+1-i}.$$
 Now if  $\lambda_{k+1}\ge d$ the second factor is $=1$
 and if $\lambda_{k+1}< d$ we have
 $\lambda'_i-\lambda'_{k+1}=\lambda_i-d< \lambda_i-\lambda_{k+1}$.
 Thus we may remove that second factor.
 Therefore
 $$
\hat{\mu_{k,n}} (\pi_m)\le 
\left( \prod_{  i\le    k+1 <j}   \frac{\lambda_i-\lambda_j+j-i}{j-i}\right)^{-1}.$$
Thus it suffices to majorize
$\left( \prod_{  i\le    k+1 <j}   \frac{\lambda_i-\lambda_j+j-i}{j-i}\right)^{-1}$
by the bounds appearing in Lemma \ref{l14}.
 We  now invoke Remark \ref{bra}. Observing that
 $n-(k+1)\le k+1$ we may apply part (ii) of Lemma \ref{l13b}
 with $k+1$ taking the place of $k$. Then replacing
 $k$ by $k+1$ in the upper bounds appearing in part (ii) in Lemma \ref{l13b}
 and using Remark \ref{bra}
we obtain the desired bounds for 
 $ \left( \prod_{  i\le    k+1 <j}   \frac{\lambda_i-\lambda_j+j-i}{j-i}\right)^{-1}.$
\end{proof}
\begin{proof}[Proof of Theorem \ref{t1}] 
We apply first part (i) in Lemma \ref{l13b} to settle the case $d=0$.
Thus we may assume $d\ge 1$.
We apply then  part (ii) from that same Lemma \ref{l13b} to settle
the cases  either $n=2k$ or $n=2k-1$, 
with the restriction $n-1>k$ which requires $n>3$. 
Then   Lemma \ref{l14} settles the remaining case $n=2k+1$.
Note that $k/n \to 1/2$  when $n\to \infty$   
if either $k=n/2$, $k=k_+$ or  $k=k_-$, and all the bounds
appearing in Lemmas \ref{l13b} and  \ref{l14}
tend to 1/4. Therefore, for any $1/4<\g<1/2$ there
is $n(\g)$ such that 
for any $n\ge n(\g)$ 
$$\sup_{\pi\not\in S_n} \| \hat{\nu_n}(\pi)\|\le \g.$$
Since $\hat{\mu_{k,n}}(\pi)=k/n$ when $\pi=\sigma_n$
 or $\pi=\ovl{\sigma_n}$  (and since $(k_+ + k_-)/2n=1/2$) we have
  $\hat{\nu_n}(\pi)=1/2$.
  Thus ${\nu_n}$ has a $(1/2,\g)$-spectral gap for any $n\ge n(\g)$,
  which
   settles the last assertion in Theorem \ref{t1}.
Checking the bounds for small values
of $n$, actually we can find a $\g<1/2$ whenever
$n\ge 4$.
\end{proof} 
\begin{rem}[A natural question] Assume that $k=[\theta n]$
where $0<\theta<1$ is fixed. Then $\hat{\mu_{k,n}} (\sigma_n)=
\hat{\mu_{k,n}} (\ovl{\sigma_n})=(n-k)/n\approx 1- \theta$.
By   Lemmas \ref{l13b} and  \ref{l14}, if we assume $\theta \le 1/2$
(to ensure that $k\le n-k$) then  
$\mu_{k,n}$ has a $( \d_n ,\g_n )$-spectral gap 
with $\d_n\approx 1-\theta$ and $\g_n\approx  (1-\theta)^2$
when $n\to \infty$. We do not know whether this (or any similar spectral gap) holds when $1/2<\theta <1$.
\end{rem}
\section{Rider's results on Sidon sets}\label{rrs}
We now turn to the applications of the spectral gap obtained in Corollary
\ref{c1}
to Sidon sets. We start with two  simple Lemmas. Their proof is not
too different from their commutative version.
\begin{lem}\label{l4} Let $G$ be any compact group. 
Let $\Lambda\subset \hat G$ be randomly Sidon with constant $C$.
Then for any   finitely supported family
$(b_\pi)$ with $b_\pi\in C(G; M_{d_\pi} )$ ($\pi\in \Lambda$) we have
  \begin{equation}\label{yy}\left|\sum\nolimits_{\pi \in \Lambda} d_\pi \tr \left(\int  {\pi(g)} b_\pi(g)  m_G(dg)\right)\right| \le C \int \left\|\sum\nolimits_{\pi \in \Lambda} d_\pi \tr({u}_\pi  b_\pi)\right\|_\infty m_{{\cl G}}(d{u}).\end{equation}
\end{lem}
\begin{proof} Let $f_\pi(t)=\int  {\pi(g)}  b_\pi(t^{-1} g) m_G(dg)$.
Then by the translation invariance of $m_G$, 
$t\mapsto     {\pi(t^{-1})} f_\pi(t)$ is constant. 
Let $$a_\pi=f_\pi(1)=\int  {\pi(g)}  b_\pi(g) m_G(dg).$$ Thus
$f_\pi(t)= {\pi(t)} f_\pi(1)= {\pi(t)}a_\pi$.
Let us write for short $\EE$ for the integral with respect to $m_{{\cl G}}$.
For any  fixed $g\in G$, by translation
invariance of the norm in $C(G)$ and
since $  ({u}_\pi )$  and $({u}_\pi  {\pi(g)} )$ have the same distribution, we have 
$$  \EE
\sup\n_{t\in G}|\sum\nolimits_{\pi \in \Lambda} d_\pi \tr({u}_\pi  b_\pi(t) )| 
=  \EE
\sup\n_{t\in G}|\sum\nolimits_{\pi \in \Lambda} d_\pi \tr({u}_\pi  {\pi(g)}  b_\pi (t^{-1} g) )|, $$
and hence
$$=\int  \EE
\sup\n_{t\in G}|\sum\nolimits_{\pi \in \Lambda} d_\pi \tr({u}_\pi  {\pi(g)}  b_\pi (t^{-1} g) )| m_G(dg)$$
and by Jensen
this is 
$$\ge \EE \sup\n_{t\in G}|\sum\nolimits_{\pi \in \Lambda} d_\pi \tr({u}_\pi f_\pi(t)| =
  \EE
\sup\n_{t\in G}|\sum\nolimits_{\pi \in \Lambda} d_\pi \tr({u}_\pi  { \pi(t)} a_\pi  )| . $$
Since $\Lambda$ is assumed randomly Sidon, 
 this last term is
$$\ge C^{-1} \sum\nolimits_{\pi \in \Lambda} d_\pi \tr |a_\pi|\ge  C^{-1}| \sum\nolimits_{\pi \in \Lambda} d_\pi \tr (a_\pi)|.$$
This completes the proof.
 \end{proof}
  \begin{rem} Let $b_\pi(g)=\pi(g^{-1}) a_\pi$. In that case
  \eqref{yy} implies
  $$|\sum\nolimits_{\pi \in \Lambda} d_\pi \tr(a_\pi)|\le C \int \|\sum\nolimits_{\pi \in \Lambda}
  d_\pi \tr(u_\pi \pi a_\pi)\|_\infty  m_{\cl G} (du).$$ 
  This shows that \eqref{yy}  generalizes   the randomly Sidon property.
\end{rem}

\begin{lem}\label{l3}\begin{itemize}
\item[(i)] Let $\Lambda
\subset \hat G$ be a Sidon set with constant $C$.  
For any ${u}\in \cl G$ (or merely for any $u\in \prod\n_{\pi\in \Lambda} U(d_\pi)$)  there is
$\mu^{u}\in M(G)$ with $\| \mu^{u} \|\le C$ such that
$\hat{\mu^{u}} (\pi)= {u}_\pi$ for any   $\pi\in \Lambda$.
\item[(ii)]  Let $\Lambda
\subset \hat G$ be a   randomly Sidon  set with constant $C$.  
Then, there is a functional $\varphi\in L_1(  {\cl G}; C(G))^*$
with norm $\le C$ such that for any $\pi\in \Lambda$ and any
$b_\pi\in C(G; M_{d_\pi} )$
$$\varphi (\tr ( {{u}_\pi}  b_\pi)) = \int \tr( {\pi(g)}  b_\pi(g))  m_G(dg). $$
\item[(iii)] Assuming (ii) and assuming $C(G)$ separable, there is a  weak* measurable (in the sense of the following remark) bounded
function ${u} \mapsto {\mu^{u}}\in M(G)$
with  $\sup \|{\mu^{u}}\|_{M(G)}\le C$
such that for any $\pi\in \Lambda$
$$\E(  {{u}_\pi} \mu^{u})=  {\pi}  m_G.$$
The latter  is an equality between matrix valued measures
(or matrices with   entries in $M(G)$)
by which we mean that for any $f\in C(G)$ we have
$$\E\left(  {{u}_\pi} \int f(g)\mu^{u}(dg)) \right)=\int f(g) {\pi(g)}  m_G(dg)=\hat f(\bar\pi).$$
\end{itemize}
\end{lem}
\begin{proof} Both (i) and (ii) are immediate consequences
of Hahn-Banach: For (i) we use the definition of Sidon sets
and for (ii) we use Lemma \ref{l4}. To check (iii),
as explained in the next remark,   we note that 
$\varphi\in L_1(  {\cl G}; C(G))^*$ defines a $\mu^{u}$ such that
$\varphi( f({u}) h(g))= \EE(f({u})  \int h(g)\mu^{u}(dg)   )$ (here $f\in L_1(\cl G)$
$h\in C(G)$),
with ${\rm ess}\sup \|{\mu^{u}}\|_{M(G)}=\|\varphi\|$.
Then
(ii) can be rephrased as saying that the action
of the $d_\pi\times d_\pi$-matrix
(with  entries in $M(G)$)
 $\E (  {{u}_\pi} \mu^{u})$  on an arbitrary $b_\pi\in C(G;M_{d_\pi}) $
 coincides with that of $ {\pi(g)}  m_G.$
 Then (iii) becomes clear.
\end{proof} 
\begin{rem}[On the dual of $L_1(  {\cl G}; C(G))$]
In the present paragraph $({\cl G},m_{\cl G})$ can be any probability space.
It is a well known fact that
$L_1(  {\cl G}; C(G))$ is the projective tensor product
of $L_1(  {\cl G})$ and $C(G)$, so that its dual can be identified isometrically
to the space $B(C(G), L_\infty({\cl G}))$ of bounded linear maps
from $ C(G)$ to $ L_\infty({\cl G})$. Explicitly,
 to any linear form
  $\varphi\in L_1(  {\cl G}; C(G))^*$   we naturally associate a bounded
  linear map $T_\varphi:\ C(G) \to L_\infty({\cl G}) $ with $\|T_\varphi\|=\|\varphi\|$
   such that
  $\varphi(f\otimes x)=  \int (T_\varphi(f)) (\omega) x(\omega) m_{\cl G}(d\omega)$
  for any $f\in C(G), x\in L_1({\cl G})$.\\
 Assume $C(G)$ separable. 
Then $\hat G$ is countable and $L_1  ({\cl G})$ is also separable.
Let $D$ be a dense countable subset of $C(G)$, and
let $V$ be its linear span. Then
any $\xi \in C(G)^*$ is determined by its values on $D$,
 and also (by linearity) by its values on $V$. Clearly we can find a measurable
  subset $\Omega_0\subset \cl G$ with full measure 
  on which all the maps  $\omega \mapsto |(T_\varphi(f)) (\omega)|$
  are bounded by $\|T_\varphi\| \|f\|$
  for any $f\in D$, and such that  $f\mapsto T_\varphi(f)(\omega)$
  extends to a  linear form of norm $\le \|T_\varphi\|$ on $C(G)$ (for this
  one way is to consider linearity over the rationals). This allows us to define
  on $\Omega_0$ a  function $\omega\mapsto \mu^\omega\in M(G)$ 
  bounded by $\|T_\varphi\|$ such that
  $\omega\mapsto \mu^\omega(f)=\int f(g) \mu^\omega(dg)$ is  measurable  for any $f\in D$ and hence for any $f\in C(G)$ (this is what we mean by ``weak* measurability") 
  with  $\sup_{\Omega_0} \|\mu^\omega\|\le \|T_\varphi\|$, that represents 
  $\varphi$ in the sense that for a.a. $\omega$
  \begin{equation}\label{zz} \int f(g) \mu^\omega(dg) =(T_\varphi(f)) (\omega).\end{equation}
  We denote by ${\cl L}_\infty({\cl G}; M(G))$ the space of
  all equivalence classes (modulo equality a.e.)
  of  bounded weak* measurable functions $\omega\mapsto \mu^\omega\in M(G)$ equipped with the norm ${\rm ess}\sup_{\omega} \|\mu^\omega\|$.
  Conversely, for any such $\omega\mapsto \mu^\omega\in M(G)$
  we can associate a bounded linear map $T:\ C(G)\to L_\infty({\cl G})$ with $\|T\|\le {\rm ess}\sup_{\omega} \|\mu^\omega\|$  that takes $f\in C(G)$ to the function $\omega\mapsto \int f(g)\mu^\omega(dg)$.
  Thus we obtain an isometric isomorphism
  between $B(C(G), L_\infty({\cl G}))$ and ${\cl L}_\infty({\cl G}; M(G))$.\\
  The preceding discussion shows that
  ${\cl L}_\infty({\cl G}; M(G))$ can be identified isometrically to
  the space $L_1(  {\cl G}; C(G))^*$.
\end{rem}
We now deduce Rider's version of Drury's Theorem :
\begin{thm}\label{t2} Let $G$ be any compact group.
Let $\Lambda\subset \hat G$ be a randomly Sidon set with constant $C$.
For any $0<\vp<1$ there is a measure $\mu_\vp\in M(G)$
  such that
  \begin{equation}\label{phr4}\sup\n_{\pi \in \Lambda}\|\hat \mu_\vp( \pi)-I\|\le \vp  \ \forall \pi \in \Lambda,\quad \sup\n_{\pi\not\in \Lambda}\|\hat \mu_\vp( \pi)\|\le \vp  \text{  and }
  \|\mu_\vp\|\le w(\vp)\end{equation}
  where $w(\vp)$ depends only on $\vp$ and $C$.\\
  More generally, for any $z\in \cl G$
  (or merely for any $z\in \prod\n_{\pi\in \Lambda} U(d_\pi)$), there is
 $ \mu_\vp^z\in M(G)$ such that
  $$\sup\n_{\pi \in \Lambda}\|\hat {\mu_\vp^z}( \pi)-z_\pi\|\le \vp  \ \forall \pi \in \Lambda,\quad \sup\n_{\pi\not\in \Lambda}\|\hat{ \mu_\vp^z}( \pi)\|\le \vp  \text{  and }
  \|\mu_\vp^z\|\le w(\vp).$$
\end{thm}
\begin{proof} We have all the ingredients to reproduce the 
Drury-Rider trick. To avoid all  irrelevant convergence 
and/or measurability issues,  we assume that $\Lambda$
is finite and that $C(G)$ is separable.
 It is easy to pass from the finite case to the general one
by a simple compactness argument (in the unit ball of $M(G)$
equipped with the weak* topology).
 Let $ \mu^{u}$ be as in
  Lemma \ref{l3} (iii).
 Let $\Lambda'\subset \hat{\cl G}$ be  the set formed
   by the coordinates $\{{u}_\pi\mid \pi \in \Lambda\}$.
   Note that here we abuse the notation: we still denote
   simply  by ${u}_\pi$
    the irreducible representation  ${u}\mapsto {u}_\pi$ on $\cl G$.
   
By Corollary \ref{c1} there is $\nu\in M(\cl G)$
with $\| \nu \|\le   w(\vp)$ such that $\hat \nu({u}_\pi)=
\int   \ovl{{u}_\pi}\nu (d{u})= I$ for $\pi\in   \Lambda$ and
$\|\hat \nu( r)\|=\|\int \bar r({u})\nu (d{u})\|\le \vp$ for any representation
$r \not\in \Lambda'$.
Let
$$\Phi^{u}=  \int  \mu^{{u}{u}'} \ast\mu^{{u}'^{-1}} m_{\cl G}(d{u}')\in {\cl L}_\infty( {\cl G} ; M(G)) .$$
Denoting $\bar z=(\ovl{z_\pi}  ) \in \cl G$, we then define
$$   \mu_\vp^z=   \int  \Phi^{{\bar z}{u} } \nu(d{u}).$$
Note
$$\| \mu_\vp^z \|\le C^2 w(\vp).$$
A simple verification  (using $({\pi}  m_G)\ast({\pi}  m_G)={\pi}  m_G$)
shows that (iii) in Lemma \ref{l3} is preserved, i.e. we have
 $$\E(  {{u}_\pi} \Phi^{u})=  {\pi}  m_G,$$
 and hence for each fixed $z\in \cl G$
 $$ \ovl{z_\pi} \E(  {{u}_\pi} \Phi^{{\bar z}u})=\E(  {({{\bar z}u)}_\pi} \Phi^{{\bar z}u})=  {\pi}  m_G.$$
Therefore    $$ \E(  {{u}_\pi} \Phi^{{\bar z}u})={}^t{z_{\pi}}   {\pi}  m_G.$$
More explicitly, for any fixed $f\in C(G)$ if we denote
$\varphi_f({u}) = \int f (g)\Phi^{u}(dg)$ we have
 \begin{equation}\label{26} \forall \pi\in\Lambda\quad\E(  {{u}_\pi}\varphi_f({{\bar z}u}) )=
{}^t{z_{\pi}}  \hat f(\bar\pi),\end{equation}
and hence taking the trace of both sides
\begin{equation}\label{26b} \forall \pi\in\Lambda\quad\E( \tr( {{u}_\pi})\varphi_f({{\bar z}u}) )=\tr(
{}^t{z_{\pi}}  \hat f(\bar\pi)).\end{equation}
By definition of 
  $ \mu^z_\vp$
 \begin{equation}\label{27}   \int f d\mu^z_\vp =\int \varphi_f ({\bar z}u)\nu(du).\end{equation}
 More generally, we can extend the definition of $\varphi_f$ to any 
 matrix-valued $f\in C(G;M_d)$: we simply set again
 $$\varphi_f ({u}) = \int f (g)\Phi^{u}(dg).$$
Let $\rho\in \hat G$.  Note that
$\varphi_{\bar \rho}= \hat{\Phi^{u}}(\rho)$.
Since $\hat{\Phi^{u}}(\rho)=  \int 
\hat{\mu^{{u}{u}'} }(\rho)
   \hat{\mu^{{u}'^{-1}} }(\rho) m_{\cl G}(d{u}')$ and ${\rm ess}\sup_{u} \|\mu^{{u}}\|\le C$, 
   the matrix valued function ${u}\mapsto \varphi_{ \rho}({u})=\hat{\Phi^{u}}(\rho)$
    (being the convolution on ${\cl G} $ of two
    $M_{d_\rho}$-valued  functions bounded by $C$)
  has its coefficients in the space of absolutely convergent Fourier series $A({\cl G} )$, so that we can apply \eqref{24} (with $\cl G$ in place of $G$)  to it. \\
  Consider the ``pseudo-measure" $\nu'$ on $\cl G$ defined a priori by its   formal Fourier expansion 
  $$\nu' = 
\sum\nolimits_{r \not\in \Lambda'} d_r \tr( {}^t\hat \nu(  r)   r ).$$
Since we assume that $\Lambda$ is finite $\nu'\in M(G)$, and 
since $\hat \nu(\pi)=I$ when $\pi\in \Lambda$ we have
\begin{equation}\label{ba} \nu = (\sum\n_{\pi\in \Lambda} d_\pi \tr( {u_\pi} ))  m_{\cl G}+\nu'.\end{equation}
Recall that by our choice of 
$\nu$ we have
$\sup\n_{r\in \hat {\cl G} } \|\hat {\nu'}(r)\|\le \vp.$
By \eqref{24} we have for any $\rho\in \hat G$
\begin{equation}\label{abc} 
\| \int \varphi_{ \rho}({\bar z}{u}) \nu'(d{u}) \|\le 
C^2\sup_{r\in \hat {\cl G} } \|\hat{\nu'}(r)\|\le C^2 \vp.
\end{equation}
 We claim that    $$\forall \pi\in \Lambda\quad  \hat{ \mu^z_\vp}( \pi)- z_\pi = \int    {\varphi_{ \pi}}({\bar z}{u}) d\nu'({u}) $$
   and
   $$\forall \rho\not\in \Lambda\quad  \hat{ \mu^z_\vp}( \rho) = \int    {\varphi_{ \rho}} ({\bar z}{u}) d\nu'({u}) .$$
From this claim and \eqref{abc} we obtain  the conclusion, except that we obtain
it with $(C^2\vp, C^2 w(\vp))$ in place of $(\vp,  w(\vp))$.\\
Thus it only  remains to justify the claim.
By \eqref{27}, \eqref{ba} and \eqref{26b} we have for any $f\in C(G)$
\begin{equation}\label{xy} \int fd{ \mu^z_\vp} -\int    {\varphi_{f}} ({\bar z}{u}) d\nu'({u}) 
=   \int    {\varphi_{f}} ({\bar z}{u})
   (\sum\n_{\pi\in \Lambda} d_\pi \tr( {u_\pi} ))dm_{\cl G}({u}) =
   \sum\n_{\pi\in \Lambda} d_\pi \tr ({}^t{z_{\pi}}  \hat{f}(\bar\pi)).
   \end{equation}
   Consider now the case
$f= \ovl {\rho_{ij}} , 1\le i,j\le d_\rho$. We have
$\hat {f }(\bar\pi)=0$ if $\rho\not=\pi  $ and  and  $\hat {f }(\bar\pi)=d_\pi^{-1} e_{ij}$ if $\rho=\pi  $. Therefore we find
$$ \sum\n_{\pi\in \Lambda} d_\pi \tr ({}^t{z_{\pi}}  \hat{f}(\bar\pi))= 1_{\rho\in \Lambda} ({z_{\pi}} )_{ij},$$
which by \eqref{xy} implies our claim.
\end{proof}
\begin{cor}[Rider, circa 1975, unpublished]\label{c2} The union of two Sidon sets is a Sidon set.
\end{cor}
\begin{proof} Let $\Lambda\subset \hat G$ be a Sidon set.
In the situation of Theorem \ref{t2}, 
for any $f\in C(G)$
we have by the triangle inequality
$\|f\ast \mu_\vp\|_\infty\ge \|\sum\nolimits_{\pi \in \Lambda} d_\pi \tr( {}^t\hat{(f\ast \mu_\vp)}(\pi) \pi  \|_\infty  - \|\sum\nolimits_{\pi\not \in \Lambda} d_\pi \tr( {}^t\hat{(f\ast \mu_\vp)}(\pi) \pi  \|_\infty$ and hence
$$w(\vp) \|f\|_\infty\ge
\|f\ast \mu_\vp\|_\infty \ge ((1-\vp)/C)\sum\nolimits_{\pi \in \Lambda} d_\pi \tr| \hat f (\pi)| -\vp \|f\|_{A(G)}.$$
 Let $\Lambda_j\subset \hat G$ be two disjoint Sidon sets
 with Sidon constants $C_j$ ($j=1,2$). 
Let $f=f_1+f_2\in C(G)$ be a function with $\hat f_j$ supported in  $\Lambda_j$.
By the preceding inequality
$$w_1(\vp) \|f\|_\infty 
 \ge (1-\vp)C^{-1}_1\sum\nolimits_{\pi \in \Lambda_1} d_\pi \tr| \hat f (\pi)| -\vp \|f_2\|_{A(G)},$$
 $$w_2(\vp) \|f\|_\infty 
 \ge (1-\vp)C^{-1}_2\sum\nolimits_{\pi \in \Lambda_2} d_\pi \tr| \hat f (\pi)| -\vp \|f_1\|_{A(G)},$$
 and hence summing both
 $$(w_1(\vp)+w_2(\vp)) \|f\|_\infty
 \ge ((1-\vp)\min\{ C^{-1}_1,C^{-1}_2\} -\vp) (\|f_1\|_{A(G)}+ \|f_2\|_{A(G)}).$$
 Then if we choose $\vp$ small enough so that 
 $C_\vp=((1-\vp)\min\{ C^{-1}_1,C^{-1}_2\} -\vp)>0$ we find
 that $\Lambda_1\cup\Lambda_2$ is Sidon with constant at most
 $(w_1(\vp)+w_2(\vp)) C_\vp^{-1}$.
\end{proof}
\begin{cor}[Rider, circa 1975, unpublished]\label{c3} Any randomly Sidon set  is a Sidon set.
\end{cor}
\begin{proof} In the situation of   Theorem \ref{t2}, for any $z=(z_\pi)\in {\cl G}$
we have for any $f\in C(G)$ with $\hat f$ supported in $\Lambda$
$$w(\vp) \|f\|_\infty
\ge \|f \ast \mu^z_\vp\|_\infty \ge
|\sum\nolimits_{\pi \in \Lambda} d_\pi  \tr( \hat f(\pi) z_\pi)| -\vp \|f\|_{A(G)} $$
and hence taking the sup over $z$
$$w(\vp) \|f\|_\infty
\ge (1-\vp) \|f\|_{A(G)} .$$
Thus, for any $\vp<1$,  $\Lambda$ is Sidon with constant at most $(1-\vp)^{-1} w(\vp)$.
\end{proof}
\begin{rem}  Actually, Corollary \ref{c3} implies Corollary \ref{c2}, because it is easy to see
that randomly Sidon sets are stable under finite unions.
\end{rem}
\begin{rem}\label{phr5}
Let $\Lambda\subset \hat G$ be Sidon with constant $C$.
Assume that for all $0<\vp<1$ there is
$\mu_\vp\in M(G)$ such that \eqref{phr4} holds.
Then $\Lambda$ is peaking. Indeed, by 
 Hahn-Banach,
for any $z\in \prod\nolimits_{\pi\in \Lambda} M_{d_\pi}$
with $\sup\nolimits_{\pi \in \Lambda}\|z_\pi\|<\infty$
there is $\nu\in M(G)$ with $\|\nu\|_{M(G)}\le C\sup\nolimits_{\pi \in \Lambda}\|z_\pi\|$
such that $\hat \nu(\pi)=z_\pi$ for any $\pi \in \Lambda$. Since
$\|\hat \mu_\vp( \pi)-I\|\le \vp <1$,  $\hat \mu_\vp( \pi)$ is invertible
and $\|(\hat \mu_\vp( \pi))^{-1}\|\le (1-\vp)^{-1}$ for any $\pi \in \Lambda$.
Let $z_\pi=(\hat \mu_\vp( \pi))^{-1}$. Let $\nu$ be the measure
(given by Hahn-Banach) such that
$\|\nu\|_{M(G)}\le C(1-\vp)^{-1}$ 
and $\hat \nu(\pi)=(\hat \mu_\vp( \pi))^{-1}$ for any $\pi \in \Lambda$.
Let $\nu_\vp=\nu \ast \mu_\vp$.
Then by \eqref{phr} $\hat \nu_\vp(\pi)=1$ for $\pi \in \Lambda$
and $\|\hat \nu_\vp(\pi)\|\le \|\nu\|_{M(G)}\|\hat \mu_\vp(\pi)\|\le
C\vp(1-\vp)^{-1}$  for $\pi \not\in \Lambda$. Also
$\|  \nu_\vp \|_{M(G)}\le \|\nu \|_{M(G)} \| \mu_\vp\|_{M(G)} \le C(1-\vp)^{-1} w(\vp)$. 
This shows that $\Lambda$ is an $\vp'$-peak set
for $\vp'=C\vp(1-\vp)^{-1}$.
By Proposition \ref{p1}
this shows that $\Lambda$ is peaking.
\end{rem}

\begin{rem} In \cite{Wi}, Wilson managed to prove the union
theorem in $\hat G$ when $G$ is a \emph{connected} compact  group.
His proof uses the structure theory of continuous compact 
groups and Lie groups. Apparently, it does not 
extend to general compact groups, and does not give any quantitative estimate.
\end{rem}

\section{Gaussian and Subgaussian random Fourier series}\label{gsr}

In this section we survey (with sketches of proofs) the main results
of \cite{Pi,Pi2}.  We will take special care of Theorem \ref{t3}
because unfortunately we detected a gap and probably  an erroneous claim
made by us in \cite{Pi2} concerning that statement (see Remark \ref{err}).

All the Gaussian variables we consider are always assumed 
(implicitly) to have mean $0$.
A Gaussian random variable $g$ will be called
normalized if   $\E|g|^2=1$. We use this
for either the real valued case or the complex valued one.
We deliberately avoid the term ``normal", which  usually
implies that $\E|g|^2=2$ in the complex case. 
By a complex valued Gaussian variable, we mean
a variable of the form $g=g_1+ig_2$ 
such that $g_1,g_2$ are  independent (real valued) Gaussian
variables with the same $L_2$-norm (and hence the same distribution).

Let $(g_n)$ be an i.i.d. sequence of real (resp.  complex) 
valued normalized Gaussian variables.
Then for any nonzero real (resp.  complex)  sequence $x=(x_n)\in \ell_2$, the variable
$g=(\sum |x_n|^2)^{-1/2} \sum x_n g_n$ is 
a normalized Gaussian variable. Therefore
 \begin{equation}\label{74}\|\sum x_n g_n\|_p= \|g_1\|_p  (\sum |x_n|^2)^{1/2}.\end{equation}
 and also in the real (resp.  complex) case
 \begin{equation}\label{75}\E \exp (\sum x_n g_n)= \exp( \sum |x_n|^2/2) \quad\text{ (resp. }  \E \exp (\Re(\sum x_n g_n))= \exp( \sum |x_n|^2/2)).\end{equation}

We now turn to the behaviour of Sidon sets in $L_p$ for $p<\infty$.
In many cases the growth of
 the $L_p$-norms of a function when $p\to \infty$ is equivalent  to its exponential integrability, as in
the following elementary and well known Lemma.

We start by recalling the definition of certain Orlicz spaces.
Let $(\Omega,\P)$ be a probability space. Let $0<a<\infty$.
Let $$\forall x\ge 0\quad \psi_a(x)=\exp {x^a}-1.$$ 
We denote by ${  L_{\psi_a} } (\P)$, or simply by $L_{\psi_a}$
the space of those $f\in L_0  (\Omega,\P)$ for which there is $t>0$ such that
 ${\bb E} \exp|f/t|^a< \infty$ and we set
\[
\|f\|_{{\psi_a} }  = \inf\{t> 0\mid {\bb E}\exp|f/t|^a\le e\}.
\]
In the next two Lemmas (and Remark \ref{p42}) we recall
several well known properties of these spaces.
\begin{lem}\label{41}
Fix a number $a>0$. The following properties of a 
(real or complex) random variable $f$ are 
equivalent:
\begin{itemize}
\item[(i)] $f\in L_p$ for all $p<\infty$ and $\sup\nolimits_{p\ge 1} p^{-1/a}\|f\|_p<\infty$.
\item[(ii)]  $f\in {L_{\psi_a} }$.
\item[(iii)] There is $t>0$ such that 
$\sup\n_{c>0} \exp{(tc^a)} \P\{ |f|>c\} <\infty$. 
\item[(iv)] Let $(f_n)$ be an i.i.d. sequence of copies of $f$. Then
$$\sup\n_n (\log(n+1))^{-1/a} |f_n| <\infty \text{   a.s.  } .$$ 
\end{itemize}
Moreover,  there is a positive constant $C_a$ such that for any $f\ge 0$ we have
 \begin{equation}\label{73}
C_a^{-1} \sup\nolimits_{p\ge 1} p^{-1/a}\|f\|_p \le \|f\|_{{\psi_a} } \le C_a \sup\nolimits_{p\ge 1} 
p^{-1/a} \|f\|_p,
\end{equation}
and this still holds if we restrict the sup over $p\ge 1$ to be over all even integers.
\end{lem}

\begin{proof} First observe that
the conditions $\sup\nolimits_{p\ge 1} p^{-1/a}\|f\|_p<\infty$
and $\sup\nolimits_{p\ge a} p^{-1/a}\|f\|_p<\infty$ are obviously equivalent.
Assume that $\sup\nolimits_{p\ge a} p^{-1/a}\|f\|_p \le 1$. Then
\begin{align*}
{\bb E} \exp|f/t|^a &= 1 + \sum\nolimits^\infty_1 {\bb E}|f/t|^{an} (n!)^{-1} 
\le 1 + \sum\nolimits^\infty_1 (an)^n t^{-an}(n!)^{-1}\\
\intertext{hence by Stirling's formula for some constant $C$}
&\le 1 + C \sum\nolimits^\infty_1(an)^n t^{-an} n^{-n}e^n=
 1 + C \sum\nolimits^\infty_1(at^{-a}e)^n
\end{align*}
from which it becomes clear (since $1<e$) that (i) implies (ii). Conversely, if (ii) holds we 
have a fortiori for all $n\ge 1$
\[
(n!)^{-1} \|f/t\|^{an}_{an} \le {\bb E} \exp|f/t|^a \le e
\]
and hence
\[
\|f\|_{an} \le e^{\frac1{an}}  (n!)^{\frac1{an}} t \le e^{\frac1{a}} n^{\frac1a} t = 
 (an)^{\frac1a} t(e/a)^{1/a},
\]
which gives $\|f\|_p \le   p^{1/a}t(e/a)^{1/a}$ for the values $p=an$, 
$n=1,2,\ldots$~. One can then easily interpolate (using H\"older's inequality)  to obtain (i). The equivalences of (ii) with (iii) and (iv) are
 elementary exercises.
The last 
assertion is   a simple recapitulation left to the reader.
\end{proof}
\begin{rem}\label{p42}
Let $$\|f\|_{\psi_a,\infty}=\inf\{t\mid \sup\n_{c>0}(\psi_a(c)\P(\{|f/t|>c\})\le \psi_a(1)\}.$$
In addition to (ii) $\Leftrightarrow$ (iii),
it is easy to check that  $\|\ \|_{\psi_a,\infty}$ and $\|\ \|_{\psi_a}$
are equivalent norms on $L_{\psi_a}$. This is in sharp contrast
with the case of $L_p$-spaces (when we replace $\psi_a$
by $c\mapsto c^p$) for which 
weak-$L_p$ is a strictly larger space than $L_p$.
\end{rem}

When  $\E f=0$  (in the case $a=2$) the following variant explains 
why the variables such that $ \|f\|_{ L_{\psi_2}}<\infty$ are usually called subGaussian.
Indeed, by \eqref{75} if $f$ is a normalized real valued Gaussian random variable,
then the number $sg(f)$ defined below is equal to 1 and equality
holds in \eqref{40} when s=1. 
Although our terminology
is slightly different, it is more customary to call subGaussian
any variable satisfying \eqref{40} below.

\begin{lem}\label{42} If $f$ is real valued, the following are equivalent:
\begin{itemize}
\item[(i)]   $f\in L_{\psi_2} $ and $\E f=0$.
\item[(ii)] There is  constant $s\ge 0$ such that
for any $t\in \R$
 \begin{equation}\label{40} \E \exp{tf}\le \exp {s^2 t^2/2}.\end{equation}
\end{itemize}
Moreover, assuming $\E f=0$,  $\|f\|_{{{\psi_2} }}$ is equivalent to the number
$sg(f)$ defined as the  smallest
$s\ge 0$ for which this holds.
\end{lem}
\begin{proof} Assume that $f\in L_{\psi_2} $ with $\|f\|_{{\psi_2}}\le 1$.
Let $f'$ be an independent copy of $f$. Let $F=f-f'$. Note that since
the distribution of  $F$ is 
  symmetric all its odd moments vanish, and hence
  $$\E \exp{xF} =1+\sum\n_{n\ge 1}  \frac{x^{2n}}{2n !} \E F^{2n}.$$
 We have $\|F\|_{\psi_2}\le  \|f\|_{\psi_2} + \|f'\|_{\psi_2}\le 2$. Therefore
 $ \E (F/2)^{2n} \le n!  \E \exp{(F /2)^2} \le  e n! $ and hence
 $$\E \exp{xF} \le 1+\sum\n_{n\ge 1}  \frac{(2x)^{2n}}{2n !} e n!\le 1+\sum\n_{n\ge 1}  \frac{(2\sqrt{e}x)^{2n}}{n !} \le \exp {(4ex^2)}. $$
 But since $t\mapsto \exp -xt$ is convex for any $x\in \R$,
 and $\E f'=0$
 we have $1=e^0 \le \E\exp -x f'$ and hence
 $\E\exp xF=\E\exp xf \E\exp - xf'\ge \E\exp xf $.
 Thus we conclude $sg(f)\le (8e)^{1/2}$. By homogeneity this
 shows $sg(f)\le (8e)^{1/2} \|f\|_{\psi_2}$.\\
 Conversely, assume $sg(f)\le 1$. Clearly \eqref{40}
 implies  $\E f=0$. Then for any $x,t>0$
 $$\P(\{ f>x\}) e^{tx}\le \E e^{tf} \le e^{x^2/2}.$$
 taking $x=t$ we find
 $\P(\{ f>t\})  \le e^{-t^2/2},$ and since $sg(-f)=sg(f)\le 1$
 we also have $\P(\{ -f>t\})  \le e^{-t^2/2}$, and hence
 $$\P(\{ |f|>t\})  \le 2 e^{-t^2/2}.$$
 Fix $c>\sqrt 2$. Let $\theta= 1/2-1/c^2$. Note $\theta>0$.
 $$\E \exp{(f/c)^2} -1 =\int_0^\infty   (2t/c^2) \exp{(t/c)^2} \P(\{ |f|>t\})   dt
 \le   \int_0^\infty   (4t/c^2)   e^{-\theta t^2}   dt= 2/\theta c^2  .$$
 Elementary calculation shows that if $c_0=({2}(e+1)(e-1)^{-1})^{1/2}$ we have
 $1+2/\theta c_0^2=e$. Thus we conclude
 $\|f\|_{\psi_2}\le c_0$.
By homogeneity, this
 shows $  \|f\|_{\psi_2} \le c_0 sg(f)$.
\end{proof}

 The next result was repeatedly used in \cite{MaPi}.
 It shows that independent random unitary matrices are dominated in a   strong sense
 by their Gaussian analogues. 
\begin{lem}\label{R72}  Let $(d_k)_{k\in I}$ be an arbitrary collection of integers.
   Let ${\bf G}=\prod_{k\in I} U(d_k)$. Let $u\mapsto u_k$
   denote the coordinates on ${\bf G}$, and $u_k(i,j)$  ($1\le i,j\le d_k$)
   the entries of $u_k$.
   Let $\{g_k(i,j)\}$ ($1\le i,j\le d_k$) be a collection of
   independent complex valued  Gaussian  random variables
  such that $\E(g_k(i,j))=0$ and $\E|g_k(i,j)|^2=1/d_k$, on a probability space
  $(\Omega,\P)$. For some $C_0>0$ there is  a positive operator
  $T: \ L_1(\Omega,\P) \to L_1({\bf G},m_{\bf G})$ with
  $ \|T: \ L_p(\Omega,\P) \to L_p({\bf G},m_{\bf G})\|\le C_0 $ for all $1\le p\le\infty$
  such that
  $$\forall k \forall  i,j\le d_k\quad T(g_k(i,j))= u_k(i,j).$$
\end{lem}
  \begin{proof}[Sketch] 
  Let $g_k=v_k |g_k|$ be the polar decomposition of $g_k$.
  Let ${\cl E}$ be the conditional expectation with respect to $(v_k)$.
 Since
    $(v_k)$ and $(|g_k|)$ are independent random variables,
    we have ${\cl E}(g_k)=v_k\E|g_k|$.
    By known results $\E|g_k|=\d_k I$
  for some $\d_k>0$ such that $\d=\inf\n_k \d_k>0$.
  Thus ${\cl E}(g_k)=v_k \d_k$. Since $0< \d/\d_k <1$ for all $k$,
  it is easy to see there is a (positive) operator
  $W:\ L_p({\bf G},m_{\bf G})\to L_p({\bf G},m_{\bf G})$ with $\|W\|\le 1$
for any $1\le p\le\infty$,
such that $W(v_k )=(\d/\d_k) v_k $ and hence
$\d^{-1} W{\cl E}(g_k )=v_k$. Thus,
since $(u_k)$ and $(v_k)$ have the same distribution,
$T=\d^{-1}W{\cl E}$ gives us the desired operator.
  \end{proof}
  
  \begin{rem}[Matricial contraction principle]\label{90} Let $(u_k)$ and $(g_k)$ be as in Lemma \ref{R72}.
  Let $\{x_k(i,j)\mid k\ge 1, 1\le i,j\le d_k\}$ be a finitely supported family 
  in an arbitrary Banach space $B$. For any matrix $a\in M_{d_k}$
  with complex entries, we denote by $ax$ and $xa$
  the matrix products (with entries in $B$)
  By convention, we write $\tr(u_k x_k)=\sum\n_{ij} u_k(i,j)x_k(j,i)$.
  With this notation, the following ``contraction principle" holds
  $$\int \|\sum d_k \tr( a_k u_k b_k x_k)\| dm_{\bf G}
 \le \sup\n_k\|a_k\|_{M_{d_k}} \sup\n_k \|b_k\|_{M_{d_k}} 
 \int \|\sum d_k \tr(   u_k    x_k)\| dm_{\bf G}.
 $$
 Indeed, this is obvious by the translation invariance of $m_{\bf G}$
 if $a_k,b_k$ are all unitary. Then the result follows   since the unit ball of ${M_{d_k}} $ is the closed convex hull of its extreme points, namely
 its unitary elements.\\
 The same inequality   holds 
 if we replace $(u_k)$  by any sequence of variables $(z_k)$  
 such that for any unitary matrices $a_k,b_k\in U(d_k)$
 the sequences $(z_k)$  and $(a_kz_kb_k)$  have the same distribution.
 In particular this holds for the Gaussian sequence $(g_k)$.
 \end{rem}
{\bf Notation:}
Let $G$ be any compact group. We denote by
  $(g_{\pi})$   an independent family indexed by $\hat G$, defined
like this:   
$g_{\pi} $ is a random ${d_\pi}\times  {d_\pi}$-matrix 
the entries of which are independent complex Gaussian
random variables with $L_2$-norm  $=(1/d_{\pi})^{1/2}$.
All our random variables are assumed defined on a suitable probability space $(\Omega, \P)$.  
\\
In the sequel, we similarly think of   $(u_\pi)$   as   an independent family of 
unitary ${d_\pi}\times  {d_\pi}$-matrices
 indexed by $\hat G$, on the probability space $({\cl G}, m_{\cl G} )$. For simplicity
 we denote the integral on $\cl G$ by $\E$. 
 
  The following basic fact  compares the notions
 of randomly Sidon for $(g_\pi)$ and $(u_\pi)$.
   It is proved by the same truncation trick
 that was used in \cite{Pi}. See \cite[Chap.V and VI]{MaPi} for further details and more general facts.   

 \begin{lem}\label{68}    For a subset $\Lambda\subset \hat G$,
 the following are equivalent:
\begin{itemize}
\item[(i)] There is a constant $\alpha_1$ such that 
for any  finitely supported family $(a_\pi)\in \prod\n_{\pi\in \Lambda} M_{d_\pi}$ 
$$\sum\n_{\pi\in \Lambda} d_\pi \tr|a_\pi| \le \alpha_1 \E\| \sum\n_{\pi\in \Lambda}  d_\pi \tr( g_\pi  \pi a_\pi)\|_\infty.$$
\item[(ii)] There is a constant $\alpha_2$ such that 
for any  finitely supported family $(a_\pi)\in \prod\n_{\pi\in \Lambda} M_{d_\pi}$ 
$$\sum\n_{\pi\in \Lambda}  d_\pi \tr|a_\pi| \le \alpha_2 \E\| \sum\n_{\pi\in \Lambda}  d_\pi \tr( u_\pi  \pi a_\pi)\|_\infty.$$
\end{itemize}
  \end{lem}
  \begin{proof}[Sketch]  
  From Lemma \ref{R72} it is easy to deduce that
  $$  \E \| \sum\n_{ \Lambda} d_\pi \tr( u_\pi  \pi a_\pi)\|_\infty m_{\bf G}(du)
  \le C_0\E\| \sum\n_{ \Lambda} d_\pi \tr( g_\pi  \pi a_\pi)\|_\infty,$$
  and hence (ii) $ \Rightarrow$ (i).
 To check the converse, recall the well known fact that 
 $c_4=\sup \E \|g_\pi\|^2 <\infty,$
 from which it is easy to deduce by Chebyshev's inequality that there
 exists   $c_5>0$ such that
 $$\sup \E(\|g_\pi\| 1_{\{\|g_\pi\|> c_5\}}  \le (2\alpha_1)^{-1}.$$
  We may assume that the sequences $(u_\pi)$
  and $(g_\pi)$ are mutually independent, so that
    the sequences $(g_\pi)$ and $(u_\pi g_\pi)$
   have the same distribution. 
Then by the triangle inequality  and by Remark \ref{90}
  $$\E\| \sum\n_{ \Lambda} d_\pi \tr( g_\pi  \pi a_\pi)\|_\infty
  =\E\|\sum\n_{ \Lambda} d_\pi \tr(u_\pi g_\pi  \pi a_\pi )\|_\infty$$
  $$
  \le \E\| \sum\n_{ \Lambda} d_\pi \tr( u_\pi g_\pi   1_{\{\|g_\pi\|\le c_5\}} \pi a_\pi )\|_\infty
  +\E\| \sum\n_{ \Lambda} d_\pi \tr( u_\pi g_\pi 1_{\{\|g_\pi\|> c_5\}} \pi a_\pi)\|_\infty$$
  $$ \le c_5 \E\| \sum\n_{ \Lambda} d_\pi \tr( u_\pi   \pi a_\pi)\|_\infty
  + \E  \sum d_\pi \|g_\pi\| 1_{\{\|g_\pi\|> c_5\}}   \tr|a_\pi|  \qquad\qquad\qquad\ \ $$
  $$ \le c_5 \E\| \sum\n_{ \Lambda} d_\pi \tr( u_\pi   \pi a_\pi)\|_\infty
  + (2\alpha_1 )^{-1}  \sum\n_{ \Lambda} d_\pi \tr|a_\pi|  . \qquad\qquad\qquad\qquad\qquad\ 
  $$
  Using this we see that 
  (i)  implies
   $$\sum\n_{ \Lambda} d_\pi \tr|a_\pi| \le \alpha_1c_5 \E\| \sum\n_{ \Lambda} d_\pi \tr( u_\pi   \pi a_\pi)\|_\infty
  +(1/2)  \sum\n_{ \Lambda} d_\pi \tr|a_\pi|  , $$
  and hence (i) $ \Rightarrow$ (ii) with $\alpha_2\le 2\alpha_1c_5$.
    \end{proof}
     
\begin{rem}[Comparison of randomizations]\label{sud} 
Actually, Lemma \ref{68} follows from a much more general fact proved in  \cite{MaPi}.   
Let $(a_\pi)$ be a finitely supported family
indexed by $\hat G$ with $a_\pi\in M_{d_\pi}$  ($\pi\in \hat G$).
In \cite{MaPi}, the random Fourier series
$$R(x) =\sum\n_{\pi\in \hat G} d_\pi \tr(  {u}_\pi \pi(x) a_\pi)\quad (x\in G)$$
randomized by $ {u}=({u}_\pi) $ on $({\cl G}, m_{\cl G} )$
is compared to
$$\tilde  R(x) =\sum\n_{\pi\in \hat G} d_\pi \tr(g_{\pi} \pi(x) a_\pi )\quad (x\in G)$$
randomized by  $g_{\pi}$ on $(\Omega, \P)$. By
 \cite[p.97]{MaPi} there is a universal constant
 $c>0$ such that
 \begin{equation}\label{mapi}c^{-1} \E \sup_{x\in G} |\tilde R(x)  |\le \E \sup_{x\in G} |R(x)  | 
 \le c \E \sup_{x\in G} |\tilde R(x)  |.\end{equation}
 In particular, a set is randomly Sidon iff it so when we replace
 the random unitaries $({u}_\pi)$ by the Gaussian variables
 $g_{\pi} $, so we recover Lemma \ref{68}. 
\end{rem}
 \begin{rem} A similar comparison holds
for 
the random Fourier series
$$L(x) =\sum d_\pi \tr({u}_\pi a_\pi \pi(x))\quad (x\in G) \text{ and  }
 \tilde  L(x) =\sum d_\pi \tr(g_{\pi} a_\pi \pi(x))\quad (x\in G),$$
 where the randomization is on the other side of $\pi$,
 but this can be easily derived from the case of $R$ and $ \tilde  R$
 by observing that
 $$ |L(x) |=|\ovl{L(x)} | =|\sum d_\pi \tr( ({u}_\pi a_\pi \pi(x))^*) |=|\sum d_\pi \tr( {u}_\pi^* \pi(x^{-1}) a_\pi^*   ) | ,$$
 and the last series can be treated as $R(x^{-1})$ for a suitable $R$.
\end{rem}
\begin{rem}
By passing to the   series $\tilde  R$, we allow ourselves
 the use of the rich  theory of Gaussian processes. 
 We will use these ideas to prove the next statement.
 Let us briefly outline this.
 Let $f(x)=  \sum\n_{\pi\in \hat G} d_\pi \tr(  \pi(x) a_\pi )$
 so that ${}^t \hat f(\pi)=a_\pi$.
 Let $f_t(g)=f(gt)$.
    Let $$d_f(s,t)=\|  \tilde R(s)-\tilde R(t)    \|_2=\|f_s-f_t\|_2=(\sum d_\pi \tr  | (\pi(s)-  \pi(t)) {}^t\hat f(\pi)|^2)^{1/2}.$$
    The metric entropy integral
    associated to $f$ is usually defined
    as $$\int_0^\infty  (\log N_f(\vp))^{1/2} d\vp$$
    where $N_f(\vp)$ is the smallest number of open balls of $d_f$-radius $\vp$
    that suffice to cover $G$. 
        \\
 Since the measure and the distance are both (left) translation invariant,
    one checks easily that
    \begin{equation}\label{ab} {m_G(\{t\mid d_f(t,1)<\vp\})}^{-1}\le N_f(\vp)\le {m_G(\{t\mid d_f(t,1)<\vp/2\})}^{-1}.\end{equation}
Thus we may  work with the following quantity 
equivalent  to the metric entropy integral :
    $${\cl I}_2(f)=\int_0^\infty  (\log \frac{1}{m_G(\{t\mid d_f(t,1)<\vp\})})^{1/2} d\vp.$$
    The metric entropy integral was originally introduced in the subject in a 1967 paper 
    of Dudley to give new upper bounds for general Gaussian processes.
In the stationary case, Fernique showed
that the same integral is also a lower bound. The latter bound
    implies  that there is an absolute constant   $c$ such that
    \begin{equation}\label{ep4}{\cl I}_2(f)
     \le c \E\| \sum  d_\pi \tr (g_\pi  \pi {}^t\hat f(\pi))\|_\infty.\end{equation}
 A fortiori,  this implies
 Sudakov's minoration (see e.g. \cite[p.69]{Piv} or \cite{Ta2}):
 there is a numerical constant $c'$
 such that
$$\sup_{\vp>0} \vp (\log N_f(\vp))^{1/2} \le c' \E \sup_{x\in G} |\tilde R(x)  |,$$ 
 and hence
\begin{equation}\label{44} \sup_{\vp>0} \vp \left(\log \frac{1}{m_G(\{x\mid d(x,1)<\vp\}}\right)^{1/2} \le cc' \E \sup_{x\in G} |  R(x)  |.\end{equation}
\end{rem}

The next two Theorems   essentially come from \cite{Pi,Pi2}. They show that
a set is Sidon  iff it is a $\Lambda(p)$-set (in Rudin's sense \cite{Ru}) for all $p>2$ with a constant growing at most like $\sqrt p$.
\begin{thm}[Sidon versus   $\Lambda(p)$-sets] \label{t6} Let $\Lambda \subset \hat G$. The following three assertions are equivalent:
\begin{itemize}
\item[(i)] $\Lambda$ is a Sidon set.
\item[(ii)] There is a constant $C$ such that for any $f\in L_2(G)$ 
with $\hat f$
supported in $\Lambda$ we have
$$\|f\|_{{{\psi_2} }} \le C \|f\|_2.$$
\item[(ii)']  There is a constant $C$ such that for any  any finitely supported family $(a_\pi)$   ($a_\pi\in M_{d_\pi}$) we have for any $p\ge 2$
$$\|\sum\n_{\pi\in \Lambda} d_\pi \tr(\pi a_\pi ) \|_p
\le Cp^{1/2} (\sum\n_{\pi\in \Lambda} d_\pi \tr|a_\pi |^2)^{1/2}.$$
\end{itemize}
 \end{thm}
   \begin{proof}[Sketch]
   The equivalence between (ii) and (ii)' is immediate by \eqref{73}.
   The proof that (i) $\Rightarrow$ (ii) follows a classical argument due to Rudin
  that Fig\`a-Talamanca and Rider adapted to the
   non-Abelian case. The quicker argument in \cite{MaPi} avoids their
   moment  computations by using instead Lemma \ref{R72}, but
   first we use (i) in Lemma \ref{l3}. 
   With the notation in that Lemma, assuming $\Lambda$  Sidon, 
     the operator of convolution
   by $\mu^{u}$ has norm $\le C$ on $L_p(G)$ for any $1\le p\le \infty$.
   Therefore, for any $f=\sum\n_{\pi\in \Lambda} d_\pi \tr ( \pi {}^t\hat f(\pi) )$ (finite sum)
   we have 
   $$\|\sum\n_{\pi\in \Lambda} 
   d_\pi \tr ( \pi  {}^t (u_\pi \hat f(\pi)) ) \|_p\le C\|f\|_p.$$
   As before let $\cl G=\prod_{\pi \in \hat G} U(d_\pi)$ (actually  we could work simply with $\prod_{\pi \in \Lambda} U(d_\pi)$).
   Let $u=(u_\pi)\in \cl G$. 
   Let  $F_u=\sum\n_{\pi\in \Lambda} d_\pi \tr (  \pi {}^t (u_\pi^*\hat f(\pi) ))$.
   Applying this with  $F_u$ in place of $f$ we find
   $$\|f\|_p\le C\|F_u\|_p$$
   and hence
    $$\|f\|_p\le C(\int \|F_u\|^p_p m_{\cl G}(du))^{1/p}.$$
   Note $F_u=\sum\n_{\pi\in \Lambda} d_\pi \tr (     u_\pi^*\hat f(\pi)    {}^t\pi )$.
   By Lemma \ref{R72}
   $$(\int \|F_u\|^p_p m_{\cl G}(du))^{1/p} \le 
   C_0 (\E\|\sum\n_{\pi\in \Lambda} d_\pi \tr (g_\pi \hat f(\pi){}^t\pi ) \|^p_p)^{1/p}
   $$ and since $({}^t\pi (x)g_\pi(\omega))$ (on $G\times \Omega$) and $(  g_\pi)$ (on $ \Omega$) have the same distribution,
   we have using \eqref{74} 
   $$(\E\|\sum\n_{\pi\in \Lambda} d_\pi \tr (g_\pi \hat f(\pi) {}^t\pi ) \|^p_p)^{1/p}
= (\E|\sum\n_{\pi\in \Lambda} d_\pi \tr (g_\pi \hat f(\pi)  ) |^p_p)^{1/p}= \gamma(p)\|f\|_2
$$
where $\gamma(p)$ is the $L_p$-norm of a normalized complex Gaussian
variable. This gives us
$$\|f\|_p\le C C_0 \gamma(p)\|f\|_2,$$ and since $\gamma(p)=O(\sqrt p)$,
       we obtain (ii) by \eqref{73}.\\
    The proof that (ii) $\Rightarrow$ (i) in \cite{Pi,MaPi} uses the metric
    entropy characterization of the
    Gaussian random Fourier series that are continuous a.s..
    We merely outline the original argument.
    Fix  $f=\sum\n_{\pi\in \Lambda} d_\pi \tr ( \pi {}^t\hat f(\pi))$ (finite sum).
      We will use Gaussian process theory through 
      the \emph{minoration} \eqref{ep4}.
But, by another result from that theory (a variant of  
 Dudley's upper bound), the integral ${\cl I}_2(f)$
 \emph{majorizes} the subGaussian processes
that are suitably dominated in the metric sense by $d_f$.
More specifically, since (ii) implies $\|f_s-f_t\|_{{\psi_2} } \le C d_f(s,t)$
the said majorization implies (assuming $\int f dm_G=0$)
that ($c'$ is here an absolute constant)
\begin{equation}\label{188}\|f\|_\infty\le c' C{\cl I}_2(f).\end{equation}
 Therefore,
we obtain
\begin{equation}\label{88'}\|f\|_\infty\le  c' C{\cl I}_2(f)\le cc' C \E\| \sum  d_\pi \tr (g_\pi  \pi {}^t \hat f(\pi) )\|_\infty, \end{equation} and hence
 \begin{equation}\label{88} |\sum  d_\pi \tr ( {}^t\hat f(\pi)  )|=|f(1)|\le cc' C \E\| \sum  d_\pi \tr (g_\pi \pi {}^t\hat f(\pi) )\|_\infty.\end{equation}
But by the distributional invariance property of $(g_\pi)$
we have
for any $z_\pi\in U(d_\pi)$
 $$\E\| \sum  d_\pi \tr (g_\pi \pi {}^t\hat f(\pi) )\|_\infty
 =\E\| \sum  d_\pi \tr (z_\pi g_\pi \pi {}^t\hat f(\pi)  )\|_\infty=\E\| \sum  d_\pi \tr (g_\pi \pi ({}^t\hat f(\pi) z_\pi) )\|_\infty,$$ and hence  
\eqref{88} applied to $\sum  d_\pi \tr (g_\pi \pi ({}^t\hat f(\pi)z_\pi) )$ implies
after taking the sup over $z_\pi$
 $$
 \sum\n_{\pi\in \Lambda} d_\pi \tr (|\hat f(\pi)|) =\sum\n_{\pi\in \Lambda} d_\pi \tr (|{}^t\hat f(\pi)|) \le cc' C \E\| \sum  d_\pi \tr (g_\pi \pi {}^t\hat f(\pi) )\|_\infty. 
$$
In other words, provided we can replace
$(g_\pi )$ by $(u_\pi )$, we  conclude that
$\Lambda$ is randomly Sidon 
and hence Sidon by Corollary \ref{c3}.
The replacement of $(g_\pi )$ by $(u_\pi )$ is   justified
by Lemma \ref{68}   (see also the discussion around \eqref{mapi}).
\end{proof}
  \begin{rem} Let $f\in C(G)$. Note $\|f\|_\infty =\int \sup_{x\in G} |f(tx)| m_G(dt)$. For proper perspective, we use this observation to rewrite \eqref{88'} as
 \begin{equation}\label{89}\|f\|_\infty =\int \|\sum  d_\pi \tr ( \pi(t) \pi {}^t \hat f(\pi) )\|_\infty m_G(dt) \le  cc' C \E\| \sum  d_\pi \tr (g_\pi \pi {}^t\hat f(\pi) )\|_\infty.\end{equation}
 Let $Y_x(t)=\sum  d_\pi \tr ( \pi(t) \pi(x) {}^t\hat f(\pi) )$
 and $X_x(\omega)=\sum  d_\pi \tr (g_\pi(\omega) \pi(x) {}^t\hat f(\pi) $.
 Then \eqref{89} means
  \begin{equation}\label{89'}\int \sup\n_{x\in G} |Y_x| dm_G \le    cc' C \E\sup\n_{x\in G}  |X_x|.\end{equation}
 In the preceding proof the  Dudley-Fernique metric entropy bounds
were used only to prove  \eqref{89} or equivalently \eqref{89'}. 
These require a certain group invariance (namely the process $(X_x)$ must be a stationary Gaussian process).
Inspired by the latter bounds, Talagrand \cite{Ta}
managed to prove a general version of \eqref{89'} that does not require any group invariance. 
More precisely, he proved that there is an 
  absolute constant $\tau_0$ such that:\\ 
  If $(\varphi_n)$ are variables such that
for any finitely supported scalar sequence $(a_n)$ we have
$$\|\sum a_n\varphi_n\|_{\psi_2}\le (\sum |x_n|^2)^{1/2}$$
and if $(f_n)$ are arbitrary functions on a set $S$ then we have
$$\E \sup\n_{x\in S} |\sum \varphi_n f_n(x)| \le \tau_0 \E\sup\n_{x\in S} |\sum g_n f_n(x)| .$$
 We use this in our recent paper \cite{Pi3} to prove a version
of the implication subGaussian $\Rightarrow$ Sidon
for general uniformly bounded orthonormal systems, that improves
an earlier   breakthrough due to Bourgain and Lewko \cite{BoLe}.
We also give in \cite{Pi3} an analogue of
(ii) $\Rightarrow$ (i) in Theorem \ref{t6} to the case when the system
$\{d_\pi^{1/2}\pi_{ij}\mid \pi\in \Lambda, 1\le i,j\le d_\pi\}$
is replaced by an orthonormal system on a probability space $(T,m)$ 
indexed by a set $\Lambda$ such that
the norms of the $d_\pi\times d_\pi$ matrices $[\pi_{ij}(t)]$
are uniformly bounded over $t\in T$ and $\pi\in \Lambda$.
In the same framework, we also give an analogue
of the equivalence between Sidon and randomly Sidon.
 See \S \ref{new} for
a related application of these ideas.
   \end{rem}
   The following refinement of Theorem \ref{t6} proved in \cite{Pi2}   will be useful.
   \begin{lem}\label{ep5} Assume that $G$ is Abelian (so that $d_\pi=1$ for all $\pi$).
   Let $1<p<2<p'<\infty$ such that $1/p+1/p'=1$. Assume that there is a constant $C$
   such that for any $f\in L_2(G)$ with $\hat f$
supported in $\Lambda$ we have
 \begin{equation}\label{ep1} \|f\|_{{{\psi_{p'}} }} \le C (\sum\n_{\pi\in \Lambda} |\hat f(\pi)|^p)^{1/p}.\end{equation}
Then $\Lambda$ is Sidon.
   \end{lem}
   \begin{proof} 
   Let  
   $$d_{p,f}(t,s)=(\sum\n_{\pi\in \hat G} | \hat f(\pi) (\pi(t)-\pi(s))  |^p )^{1/p}.$$
   We will use a variant of the metric entropy integral ${\cl I}_2(f)$, namely
   $${\cl I}_p(f)=\int_0^\infty  (\log \frac{1}{m_G(\{t\mid d_{p,f}(t,1)<\vp\})})^{1/p'} d\vp.$$
   Schematically, the proof can be described like this:
   By a generalization of the Dudley majorization \eqref{188} we have
   (assuming still $\int f dm_G=0$) that
   if we assume
   $$ \forall t,s\in G\quad \|f_t-f_s\|_{\psi_{p'}}  \le C d_{p,f}(t,s)$$
   then we have
 $$\|f\|_\infty\le C c_p' {\cl I}_p(f),$$
   and replacing $f$ by $ \sum\n_{\pi\in \hat G}  |\hat f(\pi)| \pi$ (which leaves $d_{p,f}$
   and hence also $ {\cl I}_p(f)  $   invariant) we find
    \begin{equation}\label{188p}
    \sum\n_{\pi\in \hat G} |\hat f(\pi)| \le C c_p' {\cl I}_p(f).\end{equation}
    This shows that if $\Lambda$ satisfies the assumption
    \eqref{ep1} then
  any $f\in L_2(G)$ with $\hat f$
supported in $\Lambda\setminus \{0\}$ satisfies
  \begin{equation}\label{ep2}
    \sum\n_{\pi\in \Lambda} |\hat f(\pi)| \le C c_p' {\cl I}_p(f).\end{equation}
    We may assume $0\not\in \Lambda$ for simplicity.
    The conclusion will follow from the following inequality
    \begin{equation}\label{ep3} {\cl I}_p(f) \le c'' (\sum\n_{\pi\in \hat G} |\hat f(\pi)|)^{1-\theta} {\cl I}_2(f)^{\theta} ,\end{equation}
    where $c''$ depends only on $p$ and where $0<\theta<1$.\\
    Indeed, \eqref{ep3}  combined with \eqref{ep2} implies
    $$  \sum\n_{\pi\in \Lambda} |\hat f(\pi)| \le C c_p' c'' (\sum\n_{\pi\in \hat G} |\hat f(\pi)|)^{1-\theta} {\cl I}_2(f)^{\theta}  ,$$
    and after a suitable division  we find
    $$\sum\n_{\pi\in \Lambda} |\hat f(\pi)| \le (C c_p' c'' )^{1/\theta} {\cl I}_2(f)  .$$
    But now using Fernique's lower bound \eqref{ep4}, we conclude
    as in the preceding proof that $\Lambda$ is Sidon.\\
    It remains to justify \eqref{ep3}. 
    Let $N_p(\vp)$ denote  the smallest number of sets of $d_{p,f}$-diameter 
    $\le \vp$
    that suffice to cover $G$. 
    Let $e_n(d_{p,f}) $ be the smallest
    number $\vp$ such that $G$ can be covered by $2^n$
    sets of $d_{p,f}$-diameter 
    $\le \vp$ (i.e. such that $N_p(\vp)\le 2^n$). 
    We first note that
    ${\cl I}_p(f) $ is equivalent to 
    $$\int_0^\infty  (\log N_p(\vp))^{1/{p'}} d\vp.$$
 Then,   since 
    $\int_0^\infty  (\log N_p(\vp))^{1/{p'}} d\vp=\sum_n \int_{e_n}^{e_{n-1}}  (\log N_p(\vp))^{1/{p'}} d\vp $ 
     one
    checks easily that the latter quantity
     is equivalent to the following one:
    $$\Sigma_p(f)= \sum\n _0^\infty e_n(d_{p,f}) n^{-1/{p'}}.$$
       Let  $1<q<p<2$. Let  $0<\theta<1$ be such that
    $(1-\theta)/q +\theta/2 = 1/p$. By H\"older's inequality,
    $$d_{p,f} \le d_{q,f}^{1-\theta} d_{2,f}^\theta .$$
    Thus if $A_0$ has $d_{q,f}$-diameter $\le r_0$
    and if $A_1$ has $d_{2,f}$-diameter $\le r_1$, then
    $A_0\cap A_1$ has $d_{p,f}$-diameter $\le r_0^{1-\theta} r_1^\theta$.
    From this it is clear (taking intersections) that 
    $G$ can be covered by $2^n\times 2^n$
    sets with $d_{p,f}$-diameter $\le (e_n(d_{q,f}))^{1-\theta} (e_n(d_{2,f}))^{\theta} $. In other words
    $$  e_{2n}(d_{p,f}) \le (e_n(d_{q,f}))^{1-\theta} (e_n(d_{2,f}))^{\theta} .$$
Therefore by H\"older
$$ \sum\n _0^\infty e_{2n}(d_{p,f}) n^{-1/{p'}} 
\le \sum\n _0^\infty (e_n(d_{q,f}) n^{-1/{q'}} )^{1-\theta} (e_n(d_{2,f}) n^{-1/{2}} )^{\theta}
\le (\Sigma_q(f))^{1-\theta} (\Sigma_2(f))^{\theta} .$$
But since the numbers $e_{n}(d_{p,f})$ (and also $e_{n}(d_{p,f})n^{-1/{p'}}$) are obviously non-increasing
we have
$ \sum\n _0^\infty e_{n}(d_{p,f}) n^{-1/{p'}} \le 
2\sum\n _0^\infty e_{2n}(d_{p,f}) (2n)^{-1/{p'}} $ and hence we obtain
$$\Sigma_p(f) \le 2^{1/p} (\Sigma_q(f))^{1-\theta} (\Sigma_2(f))^{\theta} .$$
Lastly, we invoke a result from  approximation theory, that tells us
that  for any $1<q<\infty$ there is a constant  $\beta_q$ 
such that  
$$\Sigma_q(f) \le \beta_q \sum\n_{\pi\in \hat G} |\hat f(\pi)|.$$
See \cite{Ma}  or   \cite[Prop. 2, p. 142]{Ca}.
 Since $\Sigma_p(f) $ is equivalent to  ${\cl I}_p(f)$, this gives us
 \eqref{ep3}.
\end{proof}
\begin{thm}[Sidon versus \emph{central} $\Lambda(p)$-sets] \label{t3} Let $\Lambda \subset \hat G$. 
Recall ${\cl G}=\prod_{\pi\in \Lambda} U(d_\pi)$. Consider the following assertions in addition to (i) and (ii) in Theorem \ref{t6}.
\begin{itemize}

\item[(iii)] Same as (ii) for all (central) functions $f $ 
of the form $f=\sum\nolimits_{\pi\in A} d_\pi\chi_\pi$ where
$A\subset \Lambda$ is an arbitrary finite subset.
\item[(iii)'] There is a constant $C$ such that for any 
even integer $2\le p<\infty$  
and any finite subset $A\subset \Lambda$ we have
$$ \|\sum\nolimits_{\pi\in A} d_\pi\chi_\pi \|_p\le C\sqrt p \|\sum\nolimits_{\pi\in A} d_\pi\chi_\pi\|_2=C\sqrt p (\sum\nolimits_{\pi\in A} d_\pi^2)^{1/2}.$$ 
  \item[(iv)] For any $0<\d<1$ there is $0<\beta<\infty$ such that for any finite
 subset $A\subset \Lambda$ we have
 $$m_G(\{t\in G\mid \sum\nolimits_{\pi\in A} d_\pi \Re(\chi_\pi) > \d \sum\nolimits_{\pi\in A} d_\pi^2\})\le e\exp{-(\beta \sum\nolimits_{\pi\in A} d_\pi^2} )$$
\item[(v)] There  are $0<\d<1$ and $0<\beta<\infty$ such that for any finite
 subset $A\subset \Lambda$ we have
 $$m_G(\{t\in G\mid \sum\nolimits_{\pi\in A} d_\pi \Re(\chi_\pi) > \d \sum\nolimits_{\pi\in A} d_\pi^2\})\le e\exp{-(\beta \sum\nolimits_{\pi\in A} d_\pi^2} )$$
 \item[(vi)]  There is a constant $C$ such that for any finite $A\subset \Lambda$
$$  \sum\nolimits_{\pi\in A} d_\pi^2
\le C \int_{\cl G} \sup\nolimits_{g\in G}   |\sum\nolimits_{\pi\in A} d_\pi    \tr ({u}_\pi  \pi(g)) | m_{\cl G}(d{u})  .$$
 \item[(vii)] There is   $0<\d<1$ such that
 any finite subset $A\subset \Lambda$ contains a further subset $B\subset A$
 with Sidon constant at most $1/\d$ and such that
 $\sum\nolimits_{\pi\in B} d^2_\pi\ge \d\sum\nolimits_{\pi\in A} d^2_\pi$.
\item[(viii)]  
 There is a constant $C$  
 such that for any finite subset $A\subset \Lambda$, 
 and any $f\in L_2(G)$ with $\hat f$
supported in $A$ we have
 \begin{equation}\label{4.17}  \| f\|_{\psi_2} \le   C  (\sum\n_{\pi\in A} d_\pi^2)^{1/2} 
 \sup\n_{\pi \in A} \|\hat f(\pi)\|.\end{equation} 
\end{itemize}
Then (i) $\Rightarrow$ (ii)  $\Rightarrow$ (iii)   $\Leftrightarrow$ (iii)'  $\Rightarrow$ (iv) $\Rightarrow$ (v)
 $\Rightarrow$ (vi) $\Rightarrow$ (vii) $\Rightarrow$ (viii).
 Moreover, (viii) $\Rightarrow$ (i) if $G$ is Abelian, or more generally
 if the dimensions $\{d_\pi\mid \pi\in \Lambda\}$ are uniformly bounded.
\end{thm}
\begin{proof} Recall (i) $\Leftrightarrow$ (ii) by Theorem \ref{t6},
  (ii) $\Rightarrow$ (iii)   is trivial and (iii) $\Leftrightarrow$ (iii)' 
follows from \eqref{73}. \\  
Assume (iii). In the rest of the proof,
we follow \cite{Pi2} except for the correction indicated in Remark \ref{err}.
 Let   $A\subset \Lambda$
be a  finite subset. Let $N(A)=\sum\nolimits_{\pi\in A} d_\pi^2$. (Incidentally,
  $N(A)$ is  the Plancherel measure of $A$.)
By \eqref{73}
$ \|\sum\nolimits_{\pi\in A} d_\pi  \chi_\pi \|_{  L_{\psi_2}}\le
 CC_2 (N(A))^{1/2}  $.
 Therefore for any $\d>0$ we have
 $$m_G(\{t\in G\mid |\sum\nolimits_{\pi\in A} d_\pi  \chi_\pi | > \d N(A)\})\le
 e\exp-(\d^2N(A)/ (CC_2)^2) .$$
A fortiori, (iv) holds and (iv) $\Rightarrow$ (v) is trivial.\\
 Assume (v).
Let   $A\subset \Lambda$
be a  finite subset.
Consider the random Fourier series
$$S_A(g) =\sum\nolimits_{\pi\in A} d_\pi \tr({u}_\pi   \pi(g))$$
defined for $ {u}=({u}_\pi)\in \cl G$
as in Remark \ref{sud}. The associated metric $d_A$ is given by
$$d_A(g,g')^2=\sum\nolimits_{\pi\in A} d_\pi \tr|  \pi(g)- \pi(g')|^2 =
2 \sum\nolimits_{\pi\in A} d^2_\pi -2 \sum\nolimits_{\pi\in A} d_\pi \Re(\chi_\pi(g'g^{-1})).$$
Therefore
$$ \{ g\in G\mid d_A(g,1)<\vp N(A)^{1/2} \}=\{g\in G\mid \sum\nolimits_{\pi\in A}d_\pi \Re(\chi_\pi(g)) >(1-\vp^2/2)N(A)\}.$$
Thus (v) implies that for some $\vp>0$ (chosen so that
 $1-\vp^2/2=\d$) we have
$$ m_G(\{ g\in G\mid d_A(g,1)<\vp N(A)^{1/2} \})
\le \exp{(1-\beta N(A) )}
.$$
Then by \eqref{44} we find
$$ \vp N(A)^{1/2} (\beta N(A)-1)^{1/2}
 \le cc' \E \sup_{g\in G} |  S_A(g)  |,$$
 from which (vi) is immediate.\\
 Assume (vi). 
 Let $V_A$ denote the linear space
 formed of all random functions of the form
 $F({u})(g)=\sum\nolimits_{\pi\in A} d_\pi    \tr ({u}_\pi a_\pi \pi(g)) $
 with $a_\pi$ arbitrary in $M_{d_\pi}$.
 Let $\|.\|_A$ be the norm  induced on it by $L_1({\cl G}; C(G))$, i.e.
 $$\|F\|_{A}=\int_{\cl G} \sup\nolimits_{g\in G } |F({u})(g) | m_{\cl G}(d{u})  .$$
 With $S_A$, $N(A)$ as before, (v) tells us that
$$\|S_A\|_{A} \ge N(A)/C.$$
By Hahn-Banach,  there is $y\in A^*$  with $\|y\|_A^*\le 1$
such that  $\langle y,S_A\rangle =\|S_A\|_{A} \ge N(A)/C.$
Identifying $y\in A^*$ with a  family $(y_\pi)$ with   $y_\pi\in M_{d_\pi}$ ($\pi\in A $),
we may assume that
$\langle y,F\rangle= \sum\nolimits_{\pi\in A} d_\pi    \tr (y_\pi a_\pi ) $.
Then $\langle y,S_A\rangle =\sum\nolimits_{\pi\in A} d_\pi    \tr (y_\pi  ) $.
Moreover, by the translation invariance 
of the norm $\| .\|_A$ (on $\cl G$ and on $G$),
for any fixed ${u}',g'$ we have
 $\|F\|_{A} =  \| F(\ . {u}' )(g'\ .) \|_A$.
 By duality this implies
 $\|y\|^*_{A} =  \| (  \pi(g') y_\pi {u}'_\pi  ) \|^*_{A}$, and hence
 $$\left\|\left( \int \pi (g' ) y_\pi  \pi({g'}^{-1})  m_G(dg') \right)\right\|^*_{A}\le 1.$$
 But since the $\pi$'s are irreducible,
 $\int \pi (g' ) y_\pi  \pi({g'}^{-1})  m_G(dg')=I_{d_\pi} \tr(y_\pi)/d_\pi  $, and hence
 $$\left\|\left( I_{d_\pi} \tr(y_\pi)/d_\pi   \right)\right\|^*_{A}\le 1,$$
 which means that for any $F$
  $$ | \sum\nolimits_{\pi\in A}   \tr (y_\pi)\tr( a_\pi )|\le \|F\|_A.$$
 Since $  \|F\|_A$ is invariant if we replace $a_\pi $ by $|a_\pi| |\tr (y_\pi)|(\tr (y_\pi))^{-1}$ we also have
 \begin{equation}\label{45}| \sum\nolimits_{\pi\in A}   |\tr (y_\pi)|\tr |a_\pi| |\le \|F\|_A. \end{equation}
 In particular, in the   case $ F({u})(g)=  d_\pi    \tr ({u}_\pi   \pi(g)) $ for some $\pi\in A$, this implies 
 \begin{equation}\label{45b}
 |d_\pi \tr (y_\pi) |\le d_\pi^2
 . \end{equation}
Now recalling  that  $\langle y,S_A\rangle =\|S_A\|_{A} \ge N(A)/C$
 we have
 $\sum\nolimits_{\pi\in A}   d_\pi \tr (y_\pi)\ge N(A)/C$, and hence there is a subset $B\subset A$ 
 (namely $B=\{\pi\mid |\tr (y_\pi)| >  d_\pi/2C\}$) such that 
$  |\sum\nolimits_{\pi\in B}   d_\pi \tr (y_\pi) | \ge  N(A)/2C$  
and
 $|\tr (y_\pi)|> d_\pi/2C$ for any $\pi\in B$. 
By \eqref{45b}
 $$\sum\nolimits_{\pi\in B}   d^2_\pi \ge  | \sum\nolimits_{\pi\in B}  d_\pi \tr (y_\pi) | \ge  N(A)/2C,$$ 
 and  by \eqref{45} the randomly Sidon constant of $B$ is at most $2C$,
 so that (vii) holds by Corollary \ref{c3}.\\
 Assume (vii). Let $N(A)=\sum\n_{\pi\in A} d_\pi^2$.
 We will show that there is   $C_\d$ depending only on the $\d$
 appearing in (vii) 
 such that for any finite subset $A\subset \Lambda$, 
 and any $f\in L_2(G)$ with $\hat f$
supported in $A$ we have
 \begin{equation}\label{1nab} \| f\|_{\psi_2} \le   C_\d N(A)^{1/2} 
 \sup\n_{\pi \in A} \|\hat f(\pi)\|.\end{equation}
 To prove this we may assume that $\Lambda $ is finite.
 Let $C_\Lambda$ be the smallest constant  
 for which \eqref{1nab} holds for all $A\subset \Lambda$. 
 
 Fix $A\subset \Lambda$ and let $B\subset A$ as in (vii).
  Let $f\in L_2(G)$ with $\hat f$
supported in $A$.
 We will show that (vii) implies that
  \begin{equation}\label{2nab} \| f\|_{\psi_2} \le  C'_\d N(A)^{1/2} \sup\n_{\pi \in A} \|\hat f(\pi)\|
 + C_\Lambda (1-\d)^{1/2} N(A)^{1/2} \sup\n_{\pi \in A} \|\hat f(\pi)\|,\end{equation}
 where $C'_\d$ is a  constant depending only on $\d$.
 Indeed, by the triangle inequality we have
 $$ \| f\|_{\psi_2} \le 
\| \sum\n_{\pi \in B}   d_\pi \tr( {}^t\hat f(\pi) \pi)   \|_{\psi_2}
 +
 \|\sum\n_{\pi \in A\setminus B}    d_\pi \tr({}^t \hat f(\pi) \pi)     \|_{\psi_2}.$$
 Note
 $$ \| \sum\n_{\pi \in B}   d_\pi \tr( {}^t\hat f(\pi) \pi)   \|_{2}  = ( \sum\n_{\pi \in B}   d_\pi \tr| \hat f(\pi)|^2  )^{1/2}\le N(B)^{1/2} \sup _{\pi \in B} \|\hat f(\pi)\|\le N(A)^{1/2} \sup\n_{\pi \in A} \|\hat f(\pi)\|.$$
Thus, by Theorem \ref{t6} applied to the set $B$ there is $C'_\d$ such that 
$$ \| \sum\n_{\pi \in B}   d_\pi \tr({}^t \hat f(\pi) \pi)   \|_{\psi_2}
\le C'_\d  \| \sum\n_{\pi \in B}   d_\pi \tr( {}^t\hat f(\pi) \pi)   \|_{2}  \le C'_\d N(A)^{1/2} \sup\n_{\pi \in A} \|\hat f(\pi)\|,
$$
and by definition of $C_\Lambda$ we have
$$ \| \sum\n_{\pi \in A\setminus B}   d_\pi \tr( {}^t\hat f(\pi) \pi)   \|_{\psi_2}
\le   C_\Lambda N(A\setminus B)^{1/2} \sup\n_{  A\setminus B} \|\hat f(\pi)\|
\le C_\Lambda (1-\d)^{1/2} N(A)^{1/2} \sup\n_{  A} \|\hat f(\pi)\|,
$$ 
from which \eqref{2nab} is immediate.\\
Equivalently,  \eqref{2nab}  means
$ C_\Lambda  \le  C'_\d + C_\Lambda (1-\d)^{1/2}
$ and hence
$$C_\Lambda  \le  ( 1- (1-\d)^{1/2} )C'_\d ,$$
which proves   \eqref{1nab}. Thus we have proved (vii) $\Rightarrow$ (viii).

Now assume (viii) but \emph{we also assume  that $d_\pi=1$ for all $\pi \in \Lambda$}.
Let us denote
by $\ell_{2,1}(\Lambda)$ the classical Lorentz space of scalar sequences
indexed by $\Lambda$.  Explicitly, given a 
scalar family $a=(a_\pi)$ (say, tending to $0$ at $\infty$), we denote 
 by $(a_n^*)$ the non-increasing rearrangement
of the numbers $\{| \hat f(\pi)  |\mid \pi \in \Lambda\}$. Let 
$$\| a \|_{2,1}=
\sum\n_1^\infty a_n^*/n^{1/2}.$$
The space $\ell_{2,1}(\Lambda)$ is defined as
formed of those  $a$ for which this sum is finite.
It is well known that $\|\ \|_{2,1}$ is equivalent to a norm on $\ell_{2,1}(\Lambda)$ (we will not use this).
Note that   \eqref{4.17}
simply means $
\| f\|_{\psi_2} \le   C  |A|^{1/2} 
 \sup\n_{ A} |\hat f(\pi)|.$
 This implies  
 $$ \| f\|_{\psi_2} \le   3C  \| (\hat f(\pi))\|_{2,1}.
 $$
 Indeed, using the disjoint decomposition
 of $\Lambda$ associated to $\{a_n^*\}=\cup_{k\ge 0} \{a_n^*\mid 2^k\le n<2^{k+1}\}$, we find
  $$ \| f\|_{\psi_2} \le   C  \sum\n_{k\ge 0}
    2^{k/2} a^*_{2^k}   \le  3 C   \sum\n_1^\infty a_n^*/n^{1/2}= 3 C   \| (\hat f(\pi))\|_{2,1}.
 $$
Let $1<p<2$. Let $2<p'<\infty $ be the conjugate, so that $1/p+   1/p'=1$.
Let
$$\| (\hat f(\pi))\|_{p}=(\sum\n_{  \Lambda} | \hat f(\pi)  |^p )^{1/p}.$$
We claim that there is a constant $\chi$ depending only on $p$  and  $C$
such that for any $f$ with $\hat f$
supported in $\Lambda$
 \begin{equation}\label{3nb} \|f\|_{\psi_{p'}} \le \chi     \| (\hat f(\pi))\|_{p}.\end{equation}This follows from a rather simple interpolation argument.
 Indeed, we have $\| (\hat f(\pi))\|_{p}=(\sum {a_n^*}^p)^{1/p}$.
 Fix a number $N\ge 1$.
 Let $f=f_0+f_1$ be the decomposition
 of $f$ associated to $\{a_n^*\}=  \{a_n^*\mid 1 \le n \le N\}  \cup  \{a_n^*\mid    n > N\}$, so that
 $$\| f_0\|_{\infty}  \le   \sum\n_1^N  a_n^*
  \text{    and   }
 \| f_1\|_{\psi_2} \le  3 C  \sum\n_{n > N} a_n^*/n^{1/2}.$$
 By homogeneity we may assume $\| (\hat f(\pi))\|_{p}=1$.
 Then ${a_n^*}\le n^{-1/p}$ for all $n\ge 1$. Therefore
 $\sum\n_1^N a_n \le \sum\n_1^N  n^{-1/p }\le p' N^{1/p'}$
 and 
 $\sum\n_{n > N} a_n^*/n^{1/2}\le \sum\n_{n > N}  n^{-1/p-1/2} \le \frac{2p'}{p'-2}N^{1/p'-1/2}$.\\ Let $c= p' N^{1/p'}$ so that $\|f_0\|_\infty \le c$.
 We have
 $$\P(\{|f|>2c\})\le \P(\{|f_0|>c\})+\P(\{|f_1|>c\}) =\P(\{|f_1|>c\}).$$
 But  we have
$$\P(\{|f_1|>c\}) \le e \exp{-c^2/\|f_1\|_{\psi_2}^2 }$$
and since $\|f_1\|_{\psi_2} \le  3 C   \frac{2p'}{p'-2}N^{1/p'-1/2} 
= 3 C   \frac{2p'}{p'-2}    (c/p')^{1-p'/2}$
  we find after substituting 
  $$\P(\{|f|>2c\})\le  \P(\{|f_1|>c\}) \le e \exp{-(\chi' c^{p'}) } ,$$
  where $\chi'$ is a constant depending only on $p$ and $C$.
  This has been established for   $c$'s
  of the form $c= p' N^{1/p'}$, but it is easy to obtain all values
  by interpolating between two such values.
  From this, our claim \eqref{3nb}
 is now  immediate (recall (iii) $\Rightarrow$ (ii) in Lemma \ref{41} and Remark \ref{p42}).
  From this claim, we obtain that $\Lambda$ is Sidon by Lemma \ref{ep5}.
  The case when the dimensions $d_\pi$ are uniformly bounded by a fixed number $D$
  follows by a straightforward modification of the same argument (but 
  all the resulting bounds will depend on $D$). In any case, this shows
   that (viii) $\Rightarrow$ (i)
  in the latter case.
\end{proof}
\begin{rem}\label{err} In \cite{Pi2} it is erroneously claimed that (viii) $\Rightarrow$ (i)
in full generality in the nonAbelian case. However we recently noticed that the proof has a serious gap, and we now believe that the result does not hold. Indeed, if 
$\Lambda=\{\pi_n\mid n\in \N\}$ and if the dimensions
of the representations in $\Lambda$ form a sequence 
such that $d^2_n\ge d^2_1+\cdots+ d^2_{n-1}$,  
  then the mere knowledge that
the individual singletons $\{\pi_n\}$  
are Sidon with a fixed constant (Rider \cite{Ri2} called this ``local Sidon property" )
is sufficient to guarantee that (vii) holds, but it seems unlikely
that this is enough to force $\Lambda$ to be Sidon.
 \end{rem}
Although we state it in full generality, the next result
is significant only if the dimensions of the irreps $\pi_n $ are unbounded.

\begin{thm}[Characterizing SubGaussian characters]\label{t4} Let $G_n$ be a sequence of compact groups, 
 let $\pi_n \in \hat{G_n}$ be nontrivial irreps and let  $\chi_{n}=\chi_{\pi_n}$ as well as
 $d_{n}=d_{\pi_n}$. The following are equivalent.
\item[(i)]   There is a constant $C$ such that the singletons $\{\pi_n\}\subset \hat G_n$
are Sidon with constant $C$, i.e. we have
$$\forall n\forall a\in M_{d_n}\quad \tr|a| \le C\sup_{g\in G} |\tr (a\pi_n(g))|.$$
\item[(ii)] There is a constant $C$ such that
$$\forall n\quad \|\chi_{ n}\| _{{{\psi_2}  }} \le C.$$
\item[(ii)'] There is $\beta>0$ such that
$$\forall n\quad \int \exp{ (\beta |\chi_{ n}|^2)} dm_{G_n}\le e.$$
\item[(ii)''] There is a constant $C$ such that for any $t\in \R$
$$\forall n\quad \int \exp{ ( t \chi_{ n} -Ct^2)} dm_{G_n}\le 1.$$
\item[(iii)]  For each $0<\delta<1$ there is $0<\theta<1$ such that
$$\forall n\quad m_{G_n} \{Re (\chi_{ n}) >\delta d_n  \} \le e \theta^{d_n^2}.$$ 
\item[(iii)']  For each $0<\delta<1$ there is $0<\theta<1$ and $D>0$ such that
for any $n$ with $d_n>D$ we have
$$   m_{G_n} \{Re (\chi_{ n}) >\delta d_n  \} \le  \theta^{d_n^2}.$$
\item[(iii)'']  There are $0<\delta<1$ and $0<\theta<1$ such that
$$\forall n\quad m_{G_n} \{Re (\chi_{ n}) >\delta d_n  \} \le e \theta^{d_n^2}.$$
\item[(iv)]There is a constant $C$ such that
$$\forall n\quad d_n\le C \int_{U(d_n)} \sup\nolimits_{g\in G_n} | \tr ({u}  \pi_n(g)) | m_{U(d_n)}(d{u})  .$$

\end{thm}
\begin{proof} Note that the properties (ii) (ii)' and (ii)'' are just reformulations of each other
by Lemmas \ref{41} and \ref{42}. 
Note that the content of (iii) and  (iii)'' is void when $ e \theta^{d_n^2}\ge 1$.
Thus (iii) $\Rightarrow$ (iii)' $\Rightarrow$ (iii)'' are trivial.
The implication (i) $\Rightarrow$ (ii) $\Rightarrow$ (iii)
is a  special case of (i) $\Rightarrow$ (ii) $\Rightarrow$ (iv) in Theorem \ref{t3}
and (iii)'' $\Rightarrow$ (iv) is a  special case  of (v) $\Rightarrow$ (vi) in Theorem \ref{t3}.
Moreover, we may invoke the implication (vi) $\Rightarrow$ (vii) in Theorem \ref{t3} for our special case of singletons.
Then the Corollary boils down to the observation that
if $\Lambda$ is a singleton the implication (vii) $\Rightarrow$ (i) in Theorem \ref{t3}
trivially holds (take $A=\Lambda$, then necessarily $B=\Lambda$). 
\end{proof}

 Although I never had  concrete examples, I believed naively for many years that  Theorem \ref{t4} could be applied to finite groups.
  To my surprise, Emmanuel Breuillard
  showed me that it is not so (and he pointed out Turing's paper
  \cite{Tu} that already emphasized that general phenomenon, back in 1938).
  It turns out that, when the groups $G_n$ are finite (or amenable as discrete groups),   Theorem \ref{t4} can hold only if the dimensions $d_n$
  remain bounded. 
  The reason lies in the presence of large Abelian subgroups
  with index of order $\exp{o(d_n^2)}$. The latter follows from the 
  quantitative refinements in  \cite{Wei,Co} of a classical Theorem of Camille Jordan
  on finite linear groups.
  See the forthcoming paper \cite{BreP}
  for details.

\section{Some questions about best constants}\label{bc}

We denote by $A_p,B_p$ the best possible constants in the classical Khintchine inequalities. These inequalities say that
 for any scalar sequence $x\in \ell_2$ we have
$$A_p(\sum |x_j|^2)^{1/2} \le  (\int |\sum \vp_j x_j |^p \ d{\bb P} )^{1/p} \le B_p (\sum |x_j|^2)^{1/2}.$$ 
After much effort by many authors,
 the exact values of $A_p,B_p$ were obtained 
 by Szarek and Haagerup (see \cite{H1+,Sz}).
Let $p_0=1.87...$ be the unique solution in the interval $]1,2[$ of  the equation $2^{1/2-1/p}=\gamma_p$
(or explicitly $\Gamma((p+1)/2)=\sqrt{\pi}/2$),   then   Haagerup (see \cite {H1+}) proved  :
\begin{equation}\label{kh-best}
A_p=2^{1/2-1/p} \quad 0<p\le p_0,
\end{equation}
\begin{equation}\label{kh-best2}
A_p= \gamma_p \quad  p_0\le p\le 2,
\end{equation}
\begin{equation}\label{kh-best3}
B_p=\gamma_p\quad  2\le p<\infty.
\end{equation}

The bounds $A_p\le   \gamma_p  $  for $p\le 2$ and $B_p\ge \gamma_p$
for $p\ge 2$ are easy consequences of the Central Limit Theorem, applied
to $\lim_{n\to \infty}(\vp_1+\cdots+\vp_n)/\sqrt n$. The bound
$A_p\le 2^{1/2-1/p}$ is immediate by
 considering  the function $(\vp_1+\vp_2)/\sqrt 2$.

For the complex analogue of these inequalities, the best constants are also
known: if we replace   the sequence $(\vp_n)$ (independent choices of signs) 
by an i.i.d. sequence $(z_n)$ uniformly distributed over
$\{z\in \C\mid |z|=1\}$, then the same inequalities hold
but now the best constants, that we denote  $A_p[\T],B_p[\T]$,
are   $A_p[\T]= \gamma^\C_p $ if   $1\le p\le 2$ and $B_p[\T]= \gamma^\C_p$ if    $p\ge 2$, where $\gamma^\C_p $ is the $L_p$-norm
of a standard complex-valued Gaussian variable normalized in $L_2$.
Indeed, in analogy with Haagerup's result, Sawa \cite{Saw,Saw2} proved that there is a phase transition
at a number $p_0^\C$, but now $0<p_0^\C=0.475...<1$!

Let $G$ be a matrix group, such as $U(d),SU(d),O(d),SO(d)$.
Let $\pi_n:\ G^\N \to G$ denote the $n$-th coordinate.
Let $E[G]$ be the linear span of the matrix coefficients
of $\Lambda=\{\pi_n\mid n\in \N\}$. Thus a typical element
of $E[G]$ can be written as a finite sum
$f=\sum \tr(\pi_n x_n)$, where $(x_n)$ is a finitely supported 
family in $M_d$. Then $\|f\|_2= (d^{-1}\sum \tr|x_n|^2)^{1/2}$.

We denote by $A_p[G],B_p[G]$  the best (positive) constants $A,B$ in the 
following inequality 
\begin{equation} \forall  f\in E[G]\qquad \label{bc1}A\|f\|_2\le \|f\|_p \le B \|f\|_2.\end{equation}

Let $\Gamma^u=\prod_{d\ge1} U(d)$,  ${\Gamma}^o=\prod_{d\ge1} O(d)$. 
We set 
${\cl G}^u= (\Gamma^u)^\N$ and ${\cl G}^o= (\Gamma^o)^\N$. We define
  similarly ${\cl G}^{su}$ and ${\cl G}^{so}$. 

\medskip 

\noindent{\bf Problem:} {\it Let $1\le p\not=2<\infty$. What are the values of $A_p[G],B_p[G]$
for $G=U(d)$ for $d>1$ ? \\ Same question for $SU(d),O(d),SO(d)$.\\
It is natural to consider also the best constants $A^c_p[G],B^c_p[G]$
for which \eqref{bc1} holds for all central functions $f$,
i.e. all $f$ of the form $f= \sum \tr(\pi_n ) x_n$ where $(x_n)$ is a finitely supported 
family in $\C$.  \\
Another natural question is to find the best $A_p[G],B_p[G]$
for $G={\cl G}^u$ and similarly when $G$ is either ${\cl G}^o$, 
${\cl G}^{su}$ or ${\cl G}^{so}$.}

The constants $A_p[{\cl G}^u],B_p[{\cl G}^u]$ can equivalently be viewed as the best constants
in \eqref{bc1} when $f$ is any finite sum of the form
$$f(\omega)= \sum \tr(\rho_n(\omega) x_n)\quad (x_n\in M_{d_n})$$
where $\Omega=\prod U(d_n)$ is equipped with its uniform (Haar) probability,
$\rho_n:\ \Omega\to U(d_n)$ is the $n$-th coordinate  and 
 $(d_n)$ is an arbitrary sequence of integers (and similarly for $o,su,so$).
Then $\|f\|_2= (\sum d_n^{-1} \tr|x_n|^2)^{1/2}$.

Consider  a (real or complex) Banach space $B$.
Recall that a $B$-valued random variable $X$ is called
Gaussian if for any real linear form $\xi: \ B\to \R$,
the real valued variable $\xi(X)$ is
Gaussian.  By definition, the covariance of a $B$-valued random variable $X$ 
is the bilinear form $(\xi,\xi')\mapsto \E(\xi(X)\xi'(X))$.  
  Let $g^{(d)}$ be a Gaussian random matrix
with the same covariance as $x\mapsto \pi_n(x)$ (the latter does not depend on $n$), so that, by the central limit theorem (CLT in short),
$n^{-1/2} (\pi_1+\cdots+\pi_n)$ tends in distribution to $g^{(d)}$.
In particular,
When $G=SO(d)$ or $O(d)$ (resp.   $G=SU(d)$ or $U(d)$) $n^{-1/2} (\tr(\pi_1)+\cdots+\tr(\pi_n))$
tends in distribution to a standard real (resp. complex) Gaussian random variable
normalized in $L_2$.  
It follows that
$B_p\ge B^c_p\ge \gamma^\R_p $ (resp. $B_p\ge B^c_p\ge \gamma^\C_p $)  for all $p\ge 2$
and $A_p\le A^c_p\le  \gamma^\R_p$
(resp. $A_p\le A^c_p\le \gamma^\C_p $)  for all $p\le 2$.
 
 In \cite[\S 36, p. 390]{HR}
 it is proved that
 $$\forall p\in 2\N\quad B_p[{\cl G}^u]\le  2 ((p/2)!)^{1/p}$$
with an improved bound for $p=4$ namely
 $B_4[{\cl G}^u] \le  2.$ A fortiori, $B_4[U(d)] \le  2$ for all $d\ge 1$.
 Since $\gamma^\C_4=2$
 this implies $$B_4[U(d)] =B^c_4[U(d)] =B_4[{\cl G}^u] =B^c_4[{\cl G}^u] = 2.$$ 
 Hewitt and Ross quote
 \cite{FTR2} but they also credit Rider and quote another paper of his
 entitled ``Continuity of random Fourier series" that apparently never appeared.
 Moreover, by a result due to Helgason \cite{Hel}
 $$A_1[{\cl G}^u] \ge  1/\sqrt{ 2}  .$$
  Let $G=U(d)$ (resp. $G=O(d)$). Let $(g_n^{(d)})$
 be an i.i.d. sequence of copies of $g^{(d)}$. 
 Following \cite{MaPi}, we describe in 
 Lemma \ref{R72} a very general comparison
principle
 showing that for some absolute constant $C_0$
 the family of coefficients $\{\pi_n(i,j)\}$
 is the image of $\{g_n^{(d)}(i,j)\}$ under a positive operator
 of norm at most $C_0$ on $L_p$.
 In  the proof of Lemma \ref{R72} we show this with
 $$C_0\le \chi =\sup\n_d  (d^{-1}\E\tr|g^{(d)}| )^{-1}<\infty,$$
 but we do not know the best value of $C_0$.
 In any case, this reasoning implies
 $$\forall p\ge 2\quad B_p[U(d)]\le
 (d^{-1}\E\tr|g^{(d)}| )^{-1} \gamma^\C_p \text{  and  }
  B_p[ {\cl G}^u] \le \chi  \gamma^\C_p  ,$$
 and similarly for $O(d)$  with the analogue of $g^{(d)}$ that has real valued Gaussian entries.
 
 \begin{rem} Let $G$ be a compact group, let $\Lambda\subset \hat G$,
 and let $E_\Lambda$ be the linear span of the matrix coefficients
 of the representations in $\Lambda$. Let $A_p^\Lambda$
 and $B_p^\Lambda$ be the best constants for which \eqref{bc1} 
holds for any $f\in E_\Lambda$.
Then, if $p>2$, $B_p^\Lambda$ can be interpreted as the
constant of $\Lambda$  as a
$\Lambda(p)$-set in Rudin's sense \cite{Ru}. See \cite{Bo}
for  a rather recent survey  on $\Lambda(p)$-sets. A similar
interpretation is valid for $A_p^\Lambda$ 
and $\Lambda(p)$-sets when $1<p<2$,
but ``true" examples of such sets are lacking for $1<p<2$.

 \end{rem}
 
 \begin{rem}
 One can also ask what are the best constants in \eqref{bc1}
 with respect to the usual non-commutative $L_p$-spaces
 when $f$ is in the linear span of free Haar unitaries
 in the sense of \cite{VDN}. 
 Now semicircular 
  (or circular) variables replace the Gaussian ones,
  when invoking the CLT, so that   $B_p\ge \|x\|_p$
  where $x$ is a semicircular (or circular)  variable 
  in the sense  of \cite{VDN} normalized in $L_2$ (note
  $\|x\|_\infty=2$). For these free Haar unitaries,  Bo\.zejko's inequality
  in  \cite{Boz}
implies
 that for any even integer
 $p=2n$ we have  $B_{2n}=(\frac{1}{n+1} {{2n}\choose{n} }  )^{1/2n}$.
 The latter number is again the $L_p$-norm of a ``free Gaussian",
 i.e. a semicircular  variable normalized in $L_2$.
  In particular (for this see also   Haagerup's     \cite{Haa1}) 
  we have $B_p\le 2$ for all $p\ge 2$. Related results appear in 
  \cite[Lemma 7]{RX2}. \\  
 See \cite{Buch2} for   interesting results on this theme.
 \end{rem}

 \section{A new approach to Rider's spectral gap}\label{new}
   
We now show how the new method presented in \cite{Pi3} 
yields another proof of Rider's spectral gap estimate.
We do not obtain the nice precise description
of the measure $\mu_{k,n}$ that possesses the desired spectral gap property,
as in Theorem \ref{t1},
but we do get a more refined quantitative bound.

We need to recall the definitions of the projective 
and injective tensor product
 norms $\|\ \|_\wedge$
 and  $\|\ \|_\vee$ on the algebraic tensor product $L_1(m_1)\otimes L_1(m_2)$ of two arbitrary $L_1$-spaces.
Let   $T=\sum x_j\otimes y_j\in L_1(m_1)\otimes L_1(m_2)$.
Then
 $$\|T\|_\wedge=\int |\sum x_j (t_1) y_j (t_2)|dm_1(t_1)dm_2(t_2)$$
 $$\|T\|_\vee=\sup \{|\sum \langle x_j ,\psi_1\rangle   \langle y_j ,\psi_2\rangle |\mid \|\psi_1\|_\infty \le 1,
 \|\psi_2\|_\infty\}.$$
 
Let $(d_k)_{k\in I}$ be an arbitrary collection of integers.
   Let $G=\prod_{k\in I} U(d_k)$.
Let $u\mapsto u_k \in U(d_k)$ denote the coordinates on $  G$.
We know that the family $\{d_k^{1/2} u_k(i,j)\}$
is subGaussian (see Lemma \ref{R72} or (i) $\Rightarrow$  (ii) in Theorem \ref{t6}).  
Let
$$S=\sum\n_{k,i,j} (d_k^{1/2} u_k(i,j))\otimes (d_k^{1/2} u_k(j,i)).$$
Actually, by Lemma \ref{R72}, in the terminology of  \cite{Pi3}, the family
$\{d_k^{1/2} u_k(i,j)\}$ is $C_0$-dominated by $\{d_k^{1/2} g_k(i,j)\}$.
Therefore,
 by \cite[Theorem 1.10]{Pi3}, for any $0<\vp<1$ there is a decomposition
 $$S= t+r$$
  for some $t,r\in L_1(G) \otimes L_1(G)$ satisfying
  $$\|t\|_{\wedge}\le w(\vp)\quad\text{and}\quad \|r\|_{\vee}\le \vp,$$
  where $w(\vp)$ depends only on $\vp$ and $w(\vp)=O(\log(1/\vp))$ when $\vp\to 0$.
  
  Consider the mapping $P:\ L_1(G) \otimes L_1(G) \to  L_1(G)$ defined by
  $P(x\otimes y)=x \ast y$.
  
  A simple verification shows that since $u_k=u_k \ast u_k$ or equivalently
  $u_k(i,j)=\sum\n_{\ell} u_k(i,\ell)\ast u_k(\ell,j)$
  
   $$P(S)=\sum\n_{k} d_k \sum\n_{i,j} u_k(i,j)\ast u_k(j,i)=\sum\n_{k} d_k \sum\n_{i} u_k(i,i) =\sum\n_{k} d_k \tr (u_k).$$
   Moreover, for any $t,r\in L_1(G) \otimes L_1(G)$ we have
   $$\|P(t)\|_1\le \|t\|_{\wedge}$$
   and 
   $$\|P(r)\|_*\le \|r\|_{\vee}$$
   where
   $$\forall f\in L_1(G)   \quad\|f\|_*=\sup\n_{\pi \in \hat G} \|\hat {f}(\pi)\|.$$
   Indeed, note
   $\|P(f)\|_*=\|T_{P(f)}\|_{B(L_2(G))}$ where
   $T_{P(f)}$ is the convolutor $ x\mapsto x\ast P(f)$. Then
by a well known consequence of Grothendieck's
 theorem (obtained using translation invariance), there is a constant $K$ 
 such that $$(K)^{-1}\|T_{P(f)}\|_{B(L_2(G))}\le \|T_{P(f)}\|_{B(L_\infty(G),L_1(G))}=
   \|f\|_\vee.$$
   Here $K$ is the complex
Grothendieck constant.  Actually (see \cite{Pig}) $K$
   is not really needed here in view of the bound $\|r\|_{\gamma_2^*}\le \vp$
   directly obtained in \cite{Pi3}).
   Indeed, if $f=\sum x_k \otimes y_k$ we have
   for any $\varphi,\psi$ in   $L_\infty(G)$ 
   $|\sum \langle\varphi,x_k\rangle \langle\psi,y_k\rangle|\le \|f\|_\vee
   \|\varphi\|_\infty  \|\psi\|_\infty$
   and hence by a suitable averaging (replacing $\varphi,\psi$ by suitable translates)
    $$\sup\n_{s,t\in G}|\sum  (\varphi \ast x_k)(s)\  (y_k\ast \psi)(t) |\le 
   K \|f\|_\vee   \|\varphi\|_2  \|\psi\|_2$$
   and hence
   $$\sup\n_{t\in G} |\sum   \varphi \ast(\sum x_k\ast y_k)\ast \psi (t) |
   \le 
   K \|f\|_\vee   \|\varphi\|_2  \|\psi\|_2$$
from which $\|P(f)\|_*=\|\sum x_k\ast y_k\|_*\le K \|f\|_\vee$ follows immediately.
   
   Thus we obtain
   \begin{thm} If the index set $I$ is finite 
   there is a decomposition 
   $\sum\n_{k\in I} d_k \tr(u_k)= T+R$
   with $T,R\in L_1(G)$ such that $\|T\|_1\le w(\vp)$ and
   $\|R\|_*\le K\vp$. If $I$ is infinite
   there is a similar decomposition within formal Fourier series
   with $T\in M(G)$ such that $\|T\|_{M(G)}\le w(\vp)$ and
   $\|R\|_*\le K\vp$. 
   \end{thm}
   Let $\pi_k(u)=u_k$ for any $u\in G$. Note $\hat T(\pi_k)=I-\hat R(\pi_k)$.
   Thus, if (say) $K\vp<1/2$ then $\|\hat T(\pi_k)-I\|\le 1/2$
  and hence $\|(\hat T(\pi_k))^{-1}\|\le  2$.
  Since the set $\Lambda=\{ \pi_k\}$ is Sidon with constant $=1$,  
  the argument in Remark \ref{phr5} shows:
    \begin{cor} For any $\vp<(2K)^{-1}$ there is a measure
    $\mu\in M(G)$ with $\|\mu\|_{M(G)}\le 2 w(\vp)$
    such that $\hat \mu(\pi_k)=I$ for all $k$
    and $\sup\n_{\pi\not\in \{\pi_k\}} \|\hat {\mu}(\pi)\|\le 2\vp$.
    
    \end{cor}
   The preceding proof yields  the estimate   $w(\vp)=O(\log(1/\vp))$
   that does not seem accessible by Rider's original approach.
   This logarithmic bound follows from
   \cite[Lemma 3]{Me}. See \cite[Remark 1.13]{Pi3} for a detailed deduction.


\begin{thebibliography}{99}
  
  
         \bibitem{Bo} J. Bourgain, Sidon sets and Riesz products. Ann. Inst. Fourier (Grenoble) 35 (1985),  
137--148.
 
   
  \bibitem{Bo2} J. Bourgain, $\Lambda_p$-sets in analysis: results, problems and related aspects. Handbook of the geometry of Banach spaces, Vol. I, 195--232, North-Holland, Amsterdam, 2001. 
  
  \bibitem{BoLe} J. Bourgain and M. Lewko, Sidonicity and variants of Kaczmarz's problem, preprint, arxiv, April 2015.
 
   \bibitem{Boz} M. Bo\. zejko, On $\Lambda(p)$ sets with minimal constant in discrete noncommutative groups,
    Proc. of the Amer. Math. Soc.   51 (1975), 407--412.
    
  \bibitem{BreP} E. Breuillard and G. Pisier, Random unitaries and amenable linear groups, in preparation.
  
  
  \bibitem{Buch2} 
A. Buchholz, Optimal constants in Khintchine type inequalities for Fermions, Rademachers and $q$-Gaussian operators, \emph{Bull. Polish Acad. Sci. Math.} {\bf 53} (2005), 315--321.

   \bibitem{Ca} B. Carl, Entropy numbers
   of diagonal operators with 
   an application to   eigenvalue problems. J. Approx. Theory 32 (1981) 135--150.
   
    \bibitem{Ce} C. Cecchini, Lacunary Fourier series on compact Lie groups. J. Funct. Anal. 11 (1972) 191--203.
   
    \bibitem{Che} S. Chevet, S\'eries de variables al\'eatoires gaussiennes \`a valeurs dans $E\otimes_\vp F$, applications aux espaces de Wiener abstraits.
    S\'eminaire sur la g\'eom\'etrie des espaces de Banach  1977-1978,
    \'Ecole Polytechnique, Exp. XIX, 1978. (available on www.numdam.org).
     
  
    \bibitem{Co} M. Collins, On Jordan's theorem for complex linear groups. J. Group Theory 10 (2007), 411--423.
 
   
 
  \bibitem{Fa} J. Faraut, \emph{Analyse sur les groupes de Lie}. Calvage \& Mounet, 2006.
  
     \bibitem{FT} A. Fig\`a-Talamanca, Random Fourier series on compact groups. Theory of Group Representations and Fourier Analysis (C.I.M.E., II Ciclo, Montecatini Terme, 1970) pp. 1--63 Edizioni Cremonese, Rome, 1971.
     
       \bibitem{FTN} A. Fig\`a-Talamanca and C. Nebbia, 
       \emph{Harmonic analysis and representation theory for groups acting on homogeneous trees},
  Cambridge University Press, Cambridge, 1991. 
       
       \bibitem{FTP} A. Fig\`a-Talamanca and M. Picardello,\emph{Harmonic analysis on free groups},  Marcel Dekker,  New York, 1983. 
       
    \bibitem{FTR1} A. Fig\`a-Talamanca and D. Rider,  
  A theorem of Littlewood and lacunary series for compact groups. Pacific J. Math. 16 (1966) 505--514. 
  
  \bibitem{FTR2} A. Fig\`a-Talamanca and D. Rider,   
   A theorem on random fourier series on
noncommutative groups. Pacific J. Math. 21 (1967) 487--492. 
   
 
   
   
   \bibitem{Fu} W. Fulton, \emph{Young tableaux}. Cambridge University Press,   1997.
   
    \bibitem{GH} 
  C. Graham and K. Hare,  
\emph{Interpolation and Sidon sets for compact groups}.
Springer, New York, 2013. xviii+249 pp. 
    
        \bibitem{GMc} C. Graham and O.C. Mc Gehee,
        \emph{Essays in commutative harmonic analysis}. 
          Springer-Verlag, New York-Berlin, 1979.   
        
         \bibitem{Haa1}   
U. Haagerup,   An example of a non-nuclear $C^*$-algebra which has the metric approximation property, \emph{Inventiones Mat.} {\bf 50} (1979), 279--293.
        
        \bibitem{H1+} 
U. Haagerup, The best constants in the Khintchine inequality,  \emph{Studia Math.} {\bf  70}  (1981), 231--283 (1982). 

   \bibitem{Hel} S. Helgason, Topologies of Group Algebras and a Theorem of Littlewood, Trans.   Amer. Math. Soc.  86  (1957), 269--283.
 
   \bibitem{HR}  E. Hewitt and K. Ross, \emph{Abstract harmonic analysis, Volume II,
Structure and Analysis for Compact Groups,
Analysis on Locally Compact Abelian Groups}, Springer,  Heidelberg, 1970.
          
     \bibitem{Hut} M. Hutchinson, Local $\Lambda$ sets for profinite groups, Pacific J.  Math.  80 (1980) 81--88.
           
           \bibitem{Ka-} J. P. Kahane, \emph{S\'eries de Fourier absolument convergentes}, Springer, 1970.
           
            \bibitem{Ka} J. P. Kahane, \emph{Some random series of functions. Second edition}, Cambridge University Press,
            1985.
  
  \bibitem{Le} M. Ledoux, \emph{ The concentration of measure phenomenon}, Mathematical Surveys and Monographs, 89. American Mathematical Society, Providence, RI, 2001.  
  
\bibitem{LeTa} M. Ledoux and M. Talagrand,  
\emph{Probability in Banach Spaces. Isoperimetry and Processes},
Springer-Verlag, Berlin, 1991. 

 \bibitem{Leh} F. Lehner, A characterization of the Leinert property. Proc. Amer. Math. Soc. 125 (1997), 3423--3431. 
   
 \bibitem{LQ} D. Li and H. Queff\'elec,  
Introduction ˆ l'\'etude des espaces de Banach.   Soci\'et\'e Math\'ematique de France, Paris, 2004.
 
  \bibitem{LiRo} J. Lindenstrauss and H.P. Rosenthal,  The $\mathcal{L}\sb{p}$ spaces,  {Israel J. Math.} { 7} (1969), 325--349.
  
 \bibitem{LR} 
J. L\'opez and K.A. Ross,  
{\it Sidon sets.}
Lecture Notes in Pure and Applied Mathematics, Vol. 13. Marcel Dekker, Inc., New York, 1975. 

 \bibitem{Ma} M.B. Marcus, The $\vp$-entropy of some compact subsets of $\ell_p$. J. Approx. Theory  10 (1974) 304--312.
 
\bibitem{MaPi}  M.B. Marcus and G. Pisier, 
{\it Random Fourier series with Applications to
Harmonic Analysis.}   Annals
of Math. Studies n$^\circ$101, Princeton Univ. Press,
1981.  
  
   \bibitem{Me}  J.-F. M\'ela, 
Mesures $\vp$-idempotentes de norme born\'ee.  
Studia Math. 72 (1982),  131--149. 

\bibitem{Park}   W. A. Parker,   Central Sidon and central $\Lambda_p$ sets. J. Austral. Math. Soc. 14,
62--74 (1972).
  
\bibitem{Pi}  G. Pisier, Ensembles de Sidon et processus gaussiens. C.R.
Acad. Sc. Paris, t. A 286 (1978) 671--674.
 
 \bibitem{Pi2}  G. Pisier, De nouvelles caract{\' e}risations des ensembles
de Sidon.  Advances in Maths. Supplementary studies, vol
7B (1981) 685--726.
 
 \bibitem{Pis} G. Pisier, Arithmetic characterizations of Sidon sets.
Bull. A.M.S. (1983)   8, 87--90.
 
 \bibitem{Pip} G. Pisier,  Probabilistic methods in the geometry of Banach spaces,  Probability and analysis (Varenna, 1985),  167--241, \emph{Lecture Notes in Math.}    {   1206}, Springer-Verlag, Berlin, 1986. 
 
   \bibitem{Piv}  G. Pisier, {\it The volume of Convex Bodies and Banach
Space Geometry}.  Cambridge University Press, 1989.

   \bibitem{Pajm}  G. Pisier, Multipliers and lacunary sets
in non amenable groups.   Amer. J. Math. 117  (1995) 337--376.
   
     \bibitem{Pig}  G. Pisier, Grothendieck's Theorem, past and present. Bull. Amer. Math. Soc. 49 (2012), 237--323.

   
    \bibitem{Pi3}  G. Pisier, On  
 uniformly  bounded orthonormal Sidon systems, preprint, arxiv 2016.
 To appear in Math. Res. Letters.
    
    
      \bibitem{Amr} A. Prasad, \emph{Representation theory. A combinatorial viewpoint.}  Cambridge University Press, Delhi, 2015.
    
    \bibitem{RX2}  E. Ricard and Q. Xu, A noncommutative martingale convexity inequality,
    Annals of Probability 44 (2016),   867--882.
    
  \bibitem{Ri}  D. Rider, Randomly continuous functions and Sidon sets.
 Duke Math. J. 42 (1975) 752--764.
 
 \bibitem{Ri2} D. Rider, $SU(n)$ has no infinite local $\Lambda_p$ sets. Boll. Un. Mat. Ital. (4) 12 (1975),   155--160.
 
  \bibitem{Ri3} D. Rider, Norms of characters and central $\Lambda_p$ sets for $U(n)$. Conference on Harmonic Analysis (Univ. Maryland, College Park, Md., 1971), pp. 287--294. Lecture Notes in Math., Vol. 266, Springer, Berlin, 1972.
 
  \bibitem{Ri4} D. Rider, Central lacunary sets. Monatsh. Math. 76 (1972), 328--338.
 
  \bibitem{Ru} W. Rudin, Trigonometric series with gaps.
   J.  Math. and Mech. 9 (1960) 203--227.
 
  \bibitem{Sag} B. Sagan, \emph{The symmetric group}
  Springer, Second edition, New-York, 2001.
  
  \bibitem{Saw} J. Sawa, The best constant in the Khintchine inequality for complex Steinhaus variables, the case $p=1$, Studia Math. 81 (1985)
  105-126.
  
    \bibitem{Saw2} J. Sawa, Some remarks on the Khintchine inequality for complex Steinhaus variables.
    
   
   \bibitem{Sta} R. Stanley, \emph{Enumerative combinatorics, vol. 2}.
   Cambridge Univ. Press
   
   \bibitem{Sz}
S. Szarek, On the best constants in the Khinchine inequality,  {Studia Math.} {\bf  58}  (1976), 197--208. 
  
 \bibitem{Ta} M. Talagrand,   Regularity of Gaussian processes. Acta Math., 159 (1987), 99--149.
 
  \bibitem{Ta2} M. Talagrand,  \emph{Upper and Lower Bounds for
Stochastic Processes}, Springer, Berlin, 2014.
  
  \bibitem{Tu} A. Turing, Finite approximations to Lie groups,
 Annals  of Math.
 39 (1938), 105--111.
 
  \bibitem{VDN} D.  Voiculescu,  K.  Dykema  and A. Nica,  \emph{Free random variables},   Amer. Math. Soc., Providence, RI, 1992. 
  
 \bibitem{Wei} B. Weisfeiler, Post-classification version of Jordan's theorem on finite
 linear groups, Proc. Natl. Acad. Sci. USA
  81 (1984), 5278--5279.
 
  \bibitem{W} H. Weyl, \emph{The classical groups}. Princeton Univ. Press,   1939.
  Reprinted by Dover.
  
  \bibitem{Wi} D. C.  Wilson,  On the structure of Sidon sets. Monatsh. Math. 101 (1986),   67--74. 

 \end{thebibliography}
 \end{document}